\documentclass[final]{siamart171218}

\usepackage{amsfonts,amssymb}
\usepackage{graphicx}
\graphicspath{{figures/}}
\usepackage{booktabs,multirow}
\usepackage{algorithm}
\usepackage{algorithmic}
\usepackage{enumerate}

\makeatletter
\def\cref@override@label@type#1#2{}
\makeatother

\newcommand{\vct}[1]{\boldsymbol{#1}}

\newsiamremark{remark}{Remark}
\newsiamthm{assumption}{Assumption}

\headers{TOBYQA for Noisy Time-Varying Derivative-Free Optimization}%
        {H. Yao and P. Xie}

\title{TOBYQA: A Time-Augmented Model-Based Method for
  Derivative-Free Optimization under Noise and Temporal
  Drift}

\author{Haoyu Yao\thanks{School of Computer Science and Technology,
    Xi'an Jiaotong University, Xi'an, China
    (\mbox{\email{yaohaoyu@stu.xjtu.edu.cn}}).}
  \and
  Pengcheng Xie\thanks{Applied Mathematics and Computational Research
    Division, Lawrence Berkeley National Laboratory, Berkeley, CA
    (\email{pxie@lbl.gov}).  Corresponding author.}}

\ifpdf
\hypersetup{
  pdftitle={TOBYQA: A Time-Augmented Model-Based Method for Derivative-Free Optimization under Noise and Temporal Drift},
  pdfauthor={H. Yao and P. Xie}
}
\fi

\begin{document}

\maketitle

\begin{abstract}
Derivative-free optimization (DFO) becomes more difficult when the
observation channel varies over time and function evaluations are
noisy. Conventional model-based methods generally assume stationary
observations; under temporal drift, previously collected data can bias
gradient estimates, and the acceptance test cannot distinguish
latent-objective decrease from temporal variation. We propose TOBYQA (Time-augmented Optimization
BY Quadratic Approximation), a regularized model-based DFO framework
that jointly incorporates spatial geometry and temporal variations
within a single saddle-point interpolation system. TOBYQA augments the
classical least-Frobenius-norm quadratic interpolation system with a
linear-in-time drift term and a ridge regularization on the residual
kernel, which accommodates observation noise and ensures
well-posedness when the constraint block has full column rank, thereby
relaxing the geometric poisedness requirements of
classical interpolation. We prove that when the temporal drift is
affine in time, the recovered gradient is algebraically invariant to
the drift rate for any noise scale and sample radius. This property
leads to a
drift-compensated acceptance test that subtracts the estimated
temporal component from the observed reduction. Driven by an
adaptive cubic regularization scheme with a closed-form step and a
geometry-guarded statistical stationarity stopping rule, TOBYQA
achieves an expected oracle complexity of $O(\varepsilon^{-2})$.
Benchmark evaluations
across diverse temporal drift regimes show that, at the intermediate
tolerance $\tau=10^{-3}$, TOBYQA solves $71.0\%$, $60.8\%$, and
$41.9\%$ of the instances at $n=6$, $10$, and $20$, respectively,
compared with $25.1\%$, $18.5\%$, and $14.9\%$ for the best-performing
comparison method at the respective dimensions. These results show
higher solve rates under temporal drift while retaining comparable
performance in static environments.
\end{abstract}



\section{Introduction}
\label{sec:intro}

In some black-box optimization problems, objective evaluations are expensive and depend on when they are performed. Examples include aerodynamic wind-tunnel testing and thermal process control, where the underlying physical configuration $\vct{x} \in \mathbb{R}^n$ may have a fixed optimum while sensor drift or temperature fluctuations change the measurement baseline during operation. Because physical trials consume time and resources, evaluation budgets are limited, and earlier evaluations cannot be repeated under identical conditions. A related setting arises in data-driven materials discovery: in the accelerated design of magnetic high-entropy alloys, Li et al.~\cite{Li2022HEA} train machine-learning surrogates of alloy properties as functions of the composition vector $\vct{x} \in \mathbb{R}^n$ and use the surrogate scores to search the composition space. When such workflows update the surrogate sequentially with newly synthesized batches, the scores supplied to the optimization loop can vary across query epochs.

Motivated by such scenarios, we consider the unconstrained minimization problem
\begin{equation}
  \label{eq:min_f}
  \min_{\vct{x} \in \mathbb{R}^n} f(\vct{x}),
\end{equation}
where the latent objective function $f:\mathbb{R}^n\to\mathbb{R}$ is continuous, smooth, and bounded from below by $f_{\mathrm{low}}>-\infty$, but cannot be directly evaluated. Instead, evaluations are obtained sequentially through a drifting, noisy observation channel. Queries are indexed by a discrete time $t \in \mathbb{N}$ that increments by one with each oracle evaluation. A query issued at point $\vct{x}$ and timestamp $t$ returns
\begin{equation}
  \label{eq:oracle_intro}
  F(\vct{x},t) \;=\; \varphi(f(\vct{x}),t) + \varepsilon(\vct{x},t),
\end{equation}
where $\varphi:\mathbb{R}\times\mathbb{N}\to\mathbb{R}$ represents an unknown temporal transformation and $\varepsilon(\vct{x},t)$ is an independent zero-mean observation noise with variance bounded by $\sigma_\varepsilon^2$. When $\varphi(\cdot, t)$ is strictly increasing with respect to its first argument, the spatial minimizers of $f$ in~\eqref{eq:min_f} remain invariant across time slices. However, because evaluations are performed sequentially, observations $F(\vct{x}_i, t_i)$ collected at distinct time steps reflect disparate temporal baselines. Finite differences and static interpolation models fitted to such data can confound the spatial gradient with temporal variation.

For stationary problems where $F$ depends solely on $\vct{x}$, derivative-free optimization (DFO) is well established~\cite{ConnScheinbergVicente2009,AudetHare2017,LarsonMenickellyWild2019}. Model-based methods such as NEWUOA~\cite{Powell2006NEWUOA} and BOBYQA~\cite{Powell2009BOBYQA} construct quadratic surrogates over sparse interpolation sets, commonly with $m=2n+1$ points, by minimizing the Frobenius norm of the model Hessian change~\cite{Powell2004LeastFrobenius}. In noisy settings, DFO-LS~\cite{CartisRoberts2019} employs overdetermined least-squares polynomial regression, while the STORM framework~\cite{ChenMenickellyScheinberg2018,BlanchetCartisMenickellyScheinberg2019} establishes global convergence under probabilistically fully-linear models. In dynamic optimization, existing methods mainly target moving optimal trajectories through gradient-based prediction--correction schemes that require temporal derivatives~\cite{Fazlyab2018,SimonettoDallAnese2020}, heuristic adaptation in evolutionary algorithms~\cite{Yazdani2025}, or temporal covariance discounting in Bayesian optimization~\cite{MoralesEnciso2015}. These approaches either use derivative information or require more samples than may be available in expensive black-box applications.

Applying classical model-based DFO to time-varying channels raises two difficulties. First, because the interpolation set is accumulated over consecutive iterations, older samples reflect earlier observation baselines; fitting a stationary polynomial surrogate to these space--time observations biases the recovered gradient. Second, the observed difference $F(\vct{x}_k, t_k) - F(\vct{x}_k+\vct{s}_k, t_k+\Delta t)$ combines latent-objective change with the channel variation accumulated over $\Delta t$. A step that improves the latent objective may therefore be rejected, while one that worsens it may be accepted. Shrinking the search radius or repeatedly re-evaluating the incumbent consumes additional queries and can reduce the local signal-to-noise ratio (SNR).

We propose TOBYQA (Time-augmented Optimization BY Quadratic Approximation), a model-based DFO algorithm that represents spatial geometry and temporal variation in one saddle-point interpolation system. Augmenting the constraint block of the least-Frobenius-norm quadratic system with a time column allows TOBYQA to recover a temporal coefficient $\gamma$ jointly with the spatial gradient and model constant from a single linear solve. Under affine temporal drift, the recovered gradient is algebraically invariant to the drift rate for any noise scale and sample radius; for general $C^2$ drift, the gradient bias is bounded by $O(L_\mu \tau_{\max}^2)$. This result motivates a drift-compensated acceptance test that subtracts the estimated temporal increment from the observed reduction. We also show that, under a rotation-invariant curvature prior, Powell's quadratic interpolation kernel is the covariance kernel of the omitted quadratic terms. This gives a generalized least-squares interpretation of the ridge regularization used to accommodate noise and relax geometric poisedness requirements. The method uses an adaptive cubic regularization scheme with an $O(n)$ closed-form step and a statistical stationarity stopping rule based on the posterior gradient covariance. Under a probabilistic fully-linear model, TOBYQA has expected oracle complexity $O(\varepsilon^{-2})$. The benchmark results report higher solve rates under the tested nonstationary channels and performance comparable to classical DFO methods on stationary problems.

\section{The TOBYQA Algorithm}
\label{sec:algorithm}

This section describes the formulation of the TOBYQA algorithm, detailing the soft-constrained quadratic model, the cubic regularization step, and the complete algorithmic procedure.

\subsection{The Soft-Constrained Quadratic Model}
\label{sec:kkt-derivation}

We start with the model-based derivative-free framework used by NEWUOA~\cite{Powell2006NEWUOA}. At each iteration, NEWUOA maintains $m$ interpolation points around the current iterate and builds a quadratic model $Q \approx f$. The coefficients of $Q$---the symmetric Hessian, the gradient, and the constant term---have $(n+1)(n+2)/2$ degrees of freedom. Determining all coefficients from the interpolation equations $Q(\vct{x}_i) = F_i$, $i = 1, \dots, (n+1)(n+2)/2$, therefore requires a number of evaluations that grows quadratically with $n$. Powell instead proposed using fewer interpolation points~\cite{Powell2004LeastFrobenius}; NEWUOA commonly uses $m = 2n+1$. The interpolation conditions determine the constant and linear coefficients and provide partial curvature information, while the remaining freedom in the quadratic part is resolved by
\[
  \min_{Q}\; \|\nabla^2 Q - \nabla^2 Q_{\mathrm{prev}}\|_F
  \quad \text{subject to} \quad Q(\vct{x}_i) = F_i, \quad i = 1, \dots, m,
\]
which selects, among the quadratics interpolating the samples, the model whose Hessian is closest to that of the previous iteration. Under standard poisedness conditions, this minimum-Frobenius-norm construction yields a unique model. The choice $m=O(n)$ also limits the temporal span of the interpolation set; Theorem~\ref{thm:nonlinear-drift} relates that span to the recovery bias under nonlinear drift.

This interpolation system does not account for the observation channel~\eqref{eq:oracle_intro}. Because the samples are collected at different times, a static quadratic model can absorb temporal variation into its spatial coefficients. In addition, the hard conditions $Q(\vct{x}_i) = F_i$ pass observation noise directly into the fitted coefficients. TOBYQA addresses both effects by adding a time column to the model and replacing hard interpolation with a penalized fit.

At iteration $k$, let $\vct{x}_k$ denote the current center and $t_k$ the timestamp at which it was evaluated. The algorithm maintains $m = 2n+1$ time-stamped samples in the active set $\mathcal{Y}_k$. For each $(\vct{x}_i,t_i,F_i)\in\mathcal{Y}_k$, define the spatial displacement $\vct{d}_i = \vct{x}_i - \vct{x}_k$ and temporal offset $\theta_i = t_i - t_k$. The center sample has $\theta_i=0$; the remaining offsets record the relative query times and may have either sign when the current center is not the most recently evaluated sample. Around the center $\vct{x}_k$, we seek a local quadratic-plus-linear-drift surrogate model
\begin{equation}\label{eq:Q-model}
  Q_k(\vct{x},t) \;=\; c + \vct{g}^{\!\top}(\vct{x} - \vct{x}_k) + \tfrac{1}{2}(\vct{x} - \vct{x}_k)^{\!\top}G\,(\vct{x} - \vct{x}_k) + \gamma\,(t - t_k),
\end{equation}
with unknowns $c\in\mathbb{R}$, $\vct{g}\in\mathbb{R}^n$, symmetric Hessian $G = G^\top\in\mathbb{R}^{n\times n}$, and temporal coefficient $\gamma\in\mathbb{R}$. Here $\gamma$ describes the local rate of change of the observation channel. Because it is estimated together with $(c,\vct{g},G)$, variation that is linear in time is represented explicitly rather than folded into the spatial model.

The model is determined by the variational problem
\begin{equation}\label{eq:soft-obj}
  \min_{c,\,\vct{g},\,G=G^\top,\,\gamma}\;
  \tfrac{1}{4}\|G\|_F^2
  \;+\;
  \tfrac{\rho}{2}\sum_{i=1}^{m}
  \bigl(Q_k(\vct{x}_i, t_i) - F_i\bigr)^2,
\end{equation}
where $\rho>0$ controls the balance between curvature regularization and data fitting. The Frobenius penalty favors a model with modest curvature, while the least-squares term accommodates noise by allowing the observations to be fitted inexactly. We regularize $G$ itself, rather than its change from the preceding model, because the preceding curvature was fitted on an earlier temporal baseline. With the factor $1/4$, the KKT reduction below produces the classical kernel $A_{ij}=\tfrac{1}{2}(\vct{d}_i^\top \vct{d}_j)^2$.

Introduce the residuals $\xi_i:=F_i-Q_k(\vct{x}_i,t_i)$. Then~\eqref{eq:soft-obj} is equivalently written as
\begin{equation}\label{eq:soft-constrained}
  \min_{c,\,\vct{g},\,G=G^\top,\,\gamma,\,\vct{\xi}}\;
  \tfrac{1}{4}\|G\|_F^2 + \tfrac{\rho}{2}\sum_{i=1}^m \xi_i^2
  \quad \text{s.t.} \quad Q_k(\vct{x}_i, t_i) + \xi_i = F_i, \;\; i = 1, \dots, m,
\end{equation}
Attaching Lagrange multipliers $\lambda_i$ to the observation constraints yields the Lagrangian
\begin{equation}\label{eq:lagrangian}
  \mathcal{L}(c, \vct{g}, G, \gamma, \vct{\xi}, \vct{\lambda})
  \;=\;
  \tfrac{1}{4}\|G\|_F^2 + \tfrac{\rho}{2}\sum_{i=1}^m \xi_i^2
  - \sum_{i=1}^m \lambda_i \Bigl(Q_k(\vct{x}_i, t_i) + \xi_i - F_i\Bigr).
\end{equation}
Setting the partial derivatives of $\mathcal{L}$ with respect to the primal variables to zero yields the KKT stationarity conditions:
\begin{align}
  \nabla_G \mathcal{L} = 0 &\implies G \;=\; \sum_{i=1}^m \lambda_i\,\vct{d}_i\,\vct{d}_i^{\!\top}, \label{eq:stat-G} \\
  \nabla_c \mathcal{L} = 0 &\implies \sum_{i=1}^m \lambda_i \;=\; 0, \label{eq:stat-c} \\
  \nabla_{\vct{g}} \mathcal{L} = 0 &\implies \sum_{i=1}^m \lambda_i \vct{d}_i \;=\; \vct{0}, \label{eq:stat-g} \\
  \nabla_\gamma \mathcal{L} = 0 &\implies \sum_{i=1}^m \lambda_i \theta_i \;=\; 0, \label{eq:stat-gamma} \\
  \nabla_{\xi_i} \mathcal{L} = 0 &\implies \rho\,\xi_i \;=\; \lambda_i \iff \xi_i \;=\; \tfrac{1}{\rho}\lambda_i. \label{eq:stat-xi}
\end{align}
The relation $\lambda_i=\rho\xi_i$ ties the multipliers directly to the residuals. The remaining conditions state that $\vct{\lambda}$ is orthogonal to the constant, spatial-linear, and temporal-linear columns of the model. In particular, $\sum_i\lambda_i\theta_i=0$ is the extra condition introduced by the time column. It separates a linear temporal component from the spatial fit and leads to the drift-decoupling result in Section~\ref{sec:properties}.

Substituting $\xi_i = \lambda_i/\rho$ and~\eqref{eq:stat-G} into the constraint equations gives, for each $i = 1, \dots, m$, the interpolation relation
\begin{equation}\label{eq:row-i}
  \tfrac{1}{2}\sum_{j=1}^m \lambda_j (\vct{d}_i^\top \vct{d}_j)^2 + c + \vct{d}_i^\top \vct{g} + \gamma \theta_i + \tfrac{1}{\rho}\lambda_i \;=\; F_i.
\end{equation}
Combining~\eqref{eq:row-i} with the three orthogonality constraints~\eqref{eq:stat-c}, \eqref{eq:stat-g}, and~\eqref{eq:stat-gamma}, we obtain the augmented KKT linear system
\begin{equation}\label{eq:tobyqa-kkt}
  \begin{bmatrix}
     A + \tfrac{1}{\rho}I_m & \mathbf{1}_m & D & \vct{\theta} \\
    \mathbf{1}_m^{\!\top} & 0 & 0 & 0 \\
    D^{\!\top} & 0 & 0 & 0 \\
    \vct{\theta}^{\!\top} & 0 & 0 & 0
  \end{bmatrix}
  \begin{bmatrix}\vct{\lambda}\\c\\\vct{g}\\\gamma\end{bmatrix}
  \;=\;
  \begin{bmatrix}\vct{F}\\0\\\vct{0}\\0\end{bmatrix},
\end{equation}
whose coefficient matrix we denote by $M$. Here $A \in \mathbb{R}^{m\times m}$ has entries $A_{ij} = \tfrac{1}{2}(\vct{d}_i^\top \vct{d}_j)^2$, $D \in \mathbb{R}^{m\times n}$ has rows $\vct{d}_i^\top$, $\vct{\theta} = (\theta_1, \dots, \theta_m)^\top \in \mathbb{R}^m$, $\mathbf{1}_m$ is the all-ones vector, and $\vct{F} = (F_1, \dots, F_m)^\top$. We denote the augmented constraint block by
\begin{equation}\label{eq:Xaug-def}
  X_{\mathrm{aug}} \;:=\;
  \begin{bmatrix}\mathbf{1}_m & D & \vct{\theta}\end{bmatrix}
  \;\in\;\mathbb{R}^{m\times(n+2)},
\end{equation}
whose first $n+1$ columns form the constraint block of the classical NEWUOA system, and whose last column $\vct{\theta}$ provides the temporal extension that enables the joint recovery of the drift rate $\gamma$.

\subsection{Cubic Regularization Step and Decision Layer}
\label{sec:arc-decision}

Given the recovered gradient $\vct{g}_k$, TOBYQA computes the trial step from
\begin{equation}\label{eq:arc-subproblem}
  \min_{\vct{s} \in \mathbb{R}^n} m_k(\vct{s}) \;:=\; \vct{g}_k^\top \vct{s} + \tfrac{\sigma_k}{3} \|\vct{s}\|_2^3,
\end{equation}
where $\sigma_k>0$ controls the step length. Problem~\eqref{eq:arc-subproblem} has the closed-form solution
\begin{equation}\label{eq:closed-form-step}
  \vct{s}_k \;=\; -\frac{\vct{g}_k}{\sqrt{\sigma_k \|\vct{g}_k\|_2}},
  \qquad
  \|\vct{s}_k\|_2 \;=\; \sqrt{\frac{\|\vct{g}_k\|_2}{\sigma_k}},
\end{equation}
whose computation requires only $O(n)$ arithmetic operations. The predicted reduction used in the acceptance test below is the decrease of the \emph{linear} part of the model, which excludes the cubic regularization term:
\begin{equation}\label{eq:pred-def}
  \mathrm{Pred}_k \;:=\; -\,\vct{g}_k^{\top} \vct{s}_k \;=\; \|\vct{g}_k\|_2^{3/2} \sigma_k^{-1/2} \;=\; \|\vct{g}_k\|_2 \|\vct{s}_k\|_2.
\end{equation}
As a safeguard against extrapolation, the trial step is additionally capped at $\chi$ times the current sample radius, $\|\vct{s}_k\|_2 \le \chi \max_i \|\vct{d}_i\|_2$; whenever the cap is active, the step is rescaled to that length and $\mathrm{Pred}_k$ is recomputed on the truncated step.

Evaluating the candidate point $\vct{x}_k + \vct{s}_k$ through the observation channel at timestamp $t_{\mathrm{new}}$ returns $F_{\mathrm{new}} := F(\vct{x}_k + \vct{s}_k, t_{\mathrm{new}})$. To account for the channel variation over the query interval $\Delta t = t_{\mathrm{new}} - t_{\mathrm{old}}$ (where $F_{\mathrm{old}} := F(\vct{x}_k, t_{\mathrm{old}})$), we use the recovered temporal coefficient $\gamma$ to define the \emph{drift-compensated actual reduction}:
\begin{equation}\label{eq:ared-def}
  \mathrm{Ared}_k \;:=\; (F_{\mathrm{old}} - F_{\mathrm{new}}) + \gamma\,(t_{\mathrm{new}} - t_{\mathrm{old}}).
\end{equation}
Since $\mathrm{Ared}_k$ contains two independent observation noise terms---$\varepsilon(\vct{x}_k, t_{\mathrm{old}})$ in $F_{\mathrm{old}}$ and $\varepsilon(\vct{x}_k + \vct{s}_k, t_{\mathrm{new}})$ in $F_{\mathrm{new}}$---its stochastic component has variance $2 R_{\mathrm{obs}}$. We therefore use the noise-gated acceptance criterion
\begin{equation}\label{eq:noise-gate}
  \mathrm{Ared}_k \;\ge\; -\frac{\kappa_{\mathrm{tol}}}{2} \sqrt{2 R_{\mathrm{obs}}},
\end{equation}
where $R_{\mathrm{obs}} = \sigma_\varepsilon^2$ is the observation noise variance and $\kappa_{\mathrm{tol}} > 0$ is a stationarity tolerance. When~\eqref{eq:noise-gate} holds, the center advances to $\vct{x}_{k+1} = \vct{x}_k + \vct{s}_k$ with $t_{k+1} = t_{\mathrm{new}}$. If the ratio $\rho_k := \mathrm{Ared}_k / \mathrm{Pred}_k$ exceeds $\eta = 0.75$, the regularization parameter is decreased according to $\sigma_{k+1} = \sigma_k / \gamma_1$, where $\gamma_1 > 1$; otherwise it is held fixed. If~\eqref{eq:noise-gate} is violated, the step is rejected, the incumbent center remains unchanged, and $\sigma_{k+1} = \gamma_2 \sigma_k$ with $\gamma_2 > 1$.

Each trial evaluation is retained in the active set. To keep its size fixed at $m=2n+1$, the sample at the current center is protected and the oldest of the remaining samples is replaced. The resulting temporal span $\tau_{\max}:=\max_i|t_i-t_k|$ appears in the nonlinear-drift error bound of Theorem~\ref{thm:nonlinear-drift}.

Finally, the algorithm terminates based on a statistical stationarity criterion derived from the posterior covariance of the KKT solve. By Proposition~\ref{prop:solvability} below, the primal parameters admit the explicit representation $(c, \vct{g}^\top, \gamma)^\top = \Pi \vct{F}$ with the recovery operator $\Pi = S^{-1} X_{\mathrm{aug}}^\top K^{-1}$, where $S = X_{\mathrm{aug}}^\top K^{-1} X_{\mathrm{aug}}$ and $K = A + \tfrac{1}{\rho} I_m$. If the observation noise were the only error source, i.e., if $\vct{F}$ carried the white covariance $R_{\mathrm{obs}} I_m$, the recovered parameters would inherit the covariance
\begin{equation}\label{eq:pnorm-def}
  \mathrm{Cov}\bigl(\Pi \vct{F}\bigr)
  \;=\; R_{\mathrm{obs}}\, S^{-1} \bigl(X_{\mathrm{aug}}^\top K^{-2} X_{\mathrm{aug}}\bigr) S^{-1}
  \;=:\; R_{\mathrm{obs}}\, P_{\mathrm{norm}},
\end{equation}
  a normalized covariance matrix that scales linearly with the noise level; writing $B := X_{\mathrm{aug}}^\top K^{-2} X_{\mathrm{aug}}$, we have $P_{\mathrm{norm}} = S^{-1} B S^{-1}$. In particular, the covariance contribution to the estimated spatial gradient is $R_{\mathrm{obs}} [P_{\mathrm{norm}}]_{gg}$. The algorithm declares convergence when the recovered gradient norm falls below the corresponding uncertainty level:
\begin{equation}\label{eq:stopping-rule}
  \|\vct{g}_k\|_2 \;<\; \kappa_{\mathrm{tol}} \sqrt{\mathrm{tr}\bigl(R_{\mathrm{obs}} [P_{\mathrm{norm}}]_{gg}\bigr)}.
\end{equation}
To avoid termination when the sample geometry temporarily inflates the trace of $P_{\mathrm{norm}}$, condition~\eqref{eq:stopping-rule} is applied only when $\mathrm{tr}(P_{\mathrm{norm}}) \le 5 \cdot \mathrm{tr}(P_{\mathrm{baseline}})$, where $P_{\mathrm{baseline}}$ is recorded at the initial orthogonal configuration. This check affects only termination. A floating-point stagnation test also terminates the run when $\|\vct{s}_k\|_2 < 10^{-14} \max(\|\vct{x}_k\|_2, 1)$.

\subsection{Algorithm Summary}
\label{sec:algo-complete}

Algorithm~\ref{alg:tobyqa} summarizes the resulting procedure. The reported implementation uses $\Delta_0=2$, $\gamma_1=2$, $\gamma_2=1.2$, $\eta=0.75$, $\kappa_{\mathrm{tol}}=3$, and $\chi=100$.

\refstepcounter{algorithm}
\label{alg:tobyqa}
\noindent\textbf{Algorithm~\thealgorithm. The TOBYQA Algorithm}
\par\nobreak\smallskip
\begin{algorithmic}[1]
\REQUIRE $\vct{x}_0\in\mathbb{R}^n$, $\Delta_0>0$, $R_{\mathrm{obs}}$, $\gamma_1,\gamma_2>1$, $\eta$, $\kappa_{\mathrm{tol}}$, $\chi>1$.
\ENSURE Approximate stationary point $\vct{x}^*$.
\STATE Evaluate $\vct{x}_0$ and $\vct{x}_0\pm\Delta_0\vct{e}_j$ $(j=1,\ldots,n)$ at successive timestamps to form $\mathcal Y_0$.
\STATE Set $F_{\mathrm{old}}\leftarrow F(\vct{x}_0,t_0)$ and $s_F\leftarrow\max_iF_i-\min_iF_i$.
\STATE Set $\widehat\tau\leftarrow\max\{s_F/\Delta_0^2,10^{-8}\}$, $R\leftarrow\max\{R_{\mathrm{obs}},(10^{-6}s_F)^2,10^{-16}\}$.
\STATE Set $\rho\leftarrow\mathrm{clip}(\widehat\tau^2/R,10^2,10^{14})$ and $\sigma_0\leftarrow\max\{s_F/\Delta_0^3,10^{-8}\}$.
\STATE Compute $P_{\mathrm{baseline}}$ from the KKT system on $\mathcal Y_0$.
\FOR{$k = 0, 1, 2, \dots$ until evaluation budget exhausted}
  \STATE Form $A$, $\vct{\theta}$, and $X_{\mathrm{aug}}$; solve~\eqref{eq:tobyqa-kkt} for $(\vct{\lambda},c,\vct{g}_k,\gamma_k)$ and compute $P_{\mathrm{norm}}$.
  \IF{$\|\vct{g}_k\|_2<\kappa_{\mathrm{tol}}\sqrt{\mathrm{tr}(R_{\mathrm{obs}}[P_{\mathrm{norm}}]_{gg})}$ and $\mathrm{tr}(P_{\mathrm{norm}})\le5\,\mathrm{tr}(P_{\mathrm{baseline}})$}
    \STATE \textbf{return} $\vct{x}_k$.
  \ENDIF
  \STATE Set $\vct{s}_k\leftarrow-\vct{g}_k/\sqrt{\sigma_k\|\vct{g}_k\|_2}$ and $\mathrm{Pred}_k\leftarrow-\vct{g}_k^\top \vct{s}_k$.
  \IF{$\|\vct{s}_k\|_2>\chi\max_i\|\vct{d}_i\|_2$}
    \STATE Rescale $\vct{s}_k$ to the cap and recompute $\mathrm{Pred}_k\leftarrow-\vct{g}_k^\top \vct{s}_k$.
  \ENDIF
  \STATE Query $(\vct{x}_k+\vct{s}_k,t_{\mathrm{new}})$ to obtain $F_{\mathrm{new}}$.
  \STATE Set $\mathrm{Ared}_k\leftarrow F_{\mathrm{old}}-F_{\mathrm{new}}+\gamma_k(t_{\mathrm{new}}-t_{\mathrm{old}})$ and $\rho_k\leftarrow\mathrm{Ared}_k/\mathrm{Pred}_k$.
  \IF{$\mathrm{Ared}_k \ge -\frac{\kappa_{\mathrm{tol}}}{2}\sqrt{2 R_{\mathrm{obs}}}$}
    \STATE $\vct{x}_{k+1}\leftarrow \vct{x}_k+\vct{s}_k$, $t_{\mathrm{old}}\leftarrow t_{\mathrm{new}}$, $F_{\mathrm{old}}\leftarrow F_{\mathrm{new}}$.
    \IF{$\rho_k > \eta$}
      \STATE $\sigma_{k+1}\leftarrow\sigma_k/\gamma_1$.
    \ELSE
      \STATE $\sigma_{k+1}\leftarrow\sigma_k$.
    \ENDIF
  \ELSE
    \STATE $\vct{x}_{k+1}\leftarrow \vct{x}_k$, $\sigma_{k+1}\leftarrow\gamma_2\sigma_k$.
  \ENDIF
  \STATE Remove from $\mathcal Y_k$ the oldest sample not located at $\vct{x}_{k+1}$.
  \STATE Insert $(\vct{x}_k+\vct{s}_k,t_{\mathrm{new}},F_{\mathrm{new}})$ to form $\mathcal Y_{k+1}$.
\ENDFOR
\end{algorithmic}

\section{Theoretical Properties of the Augmented System}
\label{sec:properties}

We now analyze the mathematical properties of the time-augmented formulation, including its algebraic solvability, drift decoupling, and statistical interpretation.

\subsection{Solvability of the Augmented KKT System}
\label{sec:well-posedness}

We begin by establishing the algebraic conditions under which the augmented saddle-point KKT system~\eqref{eq:tobyqa-kkt} possesses a unique solution. 

\begin{definition}
\label{def:rank}
The active interpolation set $\mathcal{Y}_k = \{(\vct{x}_i, t_i, F_i)\}_{i=1}^m$ is said to satisfy the augmented column-rank condition if the matrix $X_{\mathrm{aug}} = [\mathbf{1}_m \;\; D \;\; \vct{\theta}] \in \mathbb{R}^{m\times(n+2)}$ has full column rank, i.e.,
\begin{equation}\label{eq:full-rank}
  \operatorname{rank}(X_{\mathrm{aug}}) \;=\; n + 2.
\end{equation}
\end{definition}

Condition~\eqref{eq:full-rank} requires that the all-ones vector $\mathbf{1}_m$, the spatial displacement columns $\vct{d}_{\cdot, 1}, \dots, \vct{d}_{\cdot, n}$, and the temporal offset vector $\vct{\theta}$ are linearly independent in $\mathbb{R}^m$. The following proposition gives the corresponding solvability result.

\begin{proposition}
\label{prop:solvability}
Let $K := A + \tfrac{1}{\rho}I_m \in \mathbb{R}^{m\times m}$, where $A_{ij} = \tfrac{1}{2}(\vct{d}_i^\top \vct{d}_j)^2$ and $\rho > 0$. Then:
\begin{enumerate}[\rm(i)]
  \item $K$ is symmetric and strictly positive definite ($K \succ 0$).
  \item If the active set $\mathcal{Y}_k$ satisfies Definition~\ref{def:rank}, the Schur complement matrix
  \begin{equation}\label{eq:schur-S}
    S \;:=\; X_{\mathrm{aug}}^\top K^{-1} X_{\mathrm{aug}} \;\in\; \mathbb{R}^{(n+2)\times(n+2)}
  \end{equation}
  is symmetric and strictly positive definite ($S \succ 0$).
  \item The augmented KKT coefficient matrix $M$ in~\eqref{eq:tobyqa-kkt} is nonsingular, and the unique primal solution is given explicitly by
  \begin{equation}\label{eq:primal-sol}
    \begin{bmatrix}c\\\vct{g}\\\gamma\end{bmatrix}
    \;=\;
    S^{-1} X_{\mathrm{aug}}^\top K^{-1} \vct{F}
    \;=:\;
    \Pi \vct{F},
  \end{equation}
  where $\Pi := S^{-1} X_{\mathrm{aug}}^\top K^{-1} \in \mathbb{R}^{(n+2)\times m}$ is the linear recovery operator. The corresponding dual variable is $\vct{\lambda} = K^{-1}(\vct{F} - X_{\mathrm{aug}} \Pi \vct{F})$.
\end{enumerate}
\end{proposition}

\begin{proof}
For part (i), let $\vct{v} \in \mathbb{R}^m$ be non-zero. Since $A_{ij} = \tfrac{1}{2}(\vct{d}_i^\top \vct{d}_j)^2$, which equals $\tfrac{1}{2}\operatorname{tr}\bigl((\vct{d}_i \vct{d}_i^\top)(\vct{d}_j \vct{d}_j^\top)\bigr)$, the matrix $A$ is the Gram matrix of the symmetric rank-one matrices $\{\tfrac{1}{\sqrt{2}} \vct{d}_i \vct{d}_i^\top\}_{i=1}^m$ under the Frobenius inner product, whence $A$ is symmetric positive semidefinite ($A \succeq 0$). For any $\rho > 0$, we have $\vct{v}^\top K \vct{v} = \vct{v}^\top A \vct{v} + \tfrac{1}{\rho}\|\vct{v}\|_2^2 \ge \tfrac{1}{\rho}\|\vct{v}\|_2^2 > 0$, proving $K \succ 0$.

For part (ii), let $\vct{u} \in \mathbb{R}^{n+2}$ be non-zero. By Definition~\ref{def:rank}, $X_{\mathrm{aug}}$ has full column rank, so $\vct{w} := X_{\mathrm{aug}} \vct{u} \ne \vct{0} \in \mathbb{R}^m$. Because $K \succ 0$, its inverse $K^{-1} \succ 0$. Therefore, $\vct{u}^\top S \vct{u} = (X_{\mathrm{aug}} \vct{u})^\top K^{-1} (X_{\mathrm{aug}} \vct{u}) = \vct{w}^\top K^{-1} \vct{w} > 0$, which establishes $S \succ 0$.

For part (iii), the block $2\times 2$ saddle-point system~\eqref{eq:tobyqa-kkt} has invertible upper-left block $K$ and invertible Schur complement $-S$. Standard block Gaussian elimination yields the unique solution: substituting $\vct{\lambda} = K^{-1}(\vct{F} - X_{\mathrm{aug}}[c, \vct{g}^\top, \gamma]^\top)$ into $X_{\mathrm{aug}}^\top \vct{\lambda} = \vct{0}$ gives $X_{\mathrm{aug}}^\top K^{-1} X_{\mathrm{aug}}[c, \vct{g}^\top, \gamma]^\top = X_{\mathrm{aug}}^\top K^{-1} \vct{F}$, which is $S [c, \vct{g}^\top, \gamma]^\top = X_{\mathrm{aug}}^\top K^{-1} \vct{F}$. Inverting $S$ yields~\eqref{eq:primal-sol}.
\end{proof}

\begin{remark}
\label{rem:poisedness}
In classical derivative-free optimization (such as NEWUOA~\cite{Powell2006NEWUOA} or deterministic trust-region methods~\cite{ConnScheinbergVicente2009}), interpolation poisedness is used to control both solvability and conditioning. In TOBYQA, the ridge term $\tfrac{1}{\rho}I_m$ guarantees $K\succ0$, while Proposition~\ref{prop:solvability} reduces the algebraic solvability requirement to full column rank of $X_{\mathrm{aug}}$. The rank condition is weaker than strict $\Lambda$-poisedness, although near-rank-deficient configurations may still be poorly conditioned.
\end{remark}

\begin{corollary}
\label{cor:axis-rank}
Let $\mathcal{Y}_0 = \{(\vct{x}_i, t_i, F_i)\}_{i=1}^{2n+1}$ be the initial active set constructed by the coordinate axis-probe with step size $\Delta_0 > 0$ across strictly increasing timestamps $t_1 < t_2 < \dots < t_{2n+1}$, with $\vct{x}_1 = \vct{x}_0$ evaluated at $t_1$. Then $\operatorname{rank}(X_{\mathrm{aug}}) = n + 2$.
\end{corollary}

\begin{proof}
In the displacement matrix $D \in \mathbb{R}^{(2n+1)\times n}$, the rows corresponding to the probe pair $\vct{x}_0 \pm \Delta_0 \vct{e}_j$ are $\pm\Delta_0 \vct{e}_j^\top$, so the $n$ columns of $D$ are linearly independent and $\operatorname{rank}(D) = n$; moreover, the all-ones column is independent of the columns of $D$, because the row of $D$ at the center sample is zero while the corresponding entry of $\mathbf{1}_{2n+1}$ is one. Hence $\operatorname{rank}([\mathbf{1}_{2n+1} \;\; D]) = n+1$.

Suppose, for contradiction, that $\vct{\theta} \in \operatorname{span}\{\mathbf{1}_{2n+1}, D\}$, i.e., $\theta_i = \alpha + \vct{\beta}^\top \vct{d}_i$ for some $\alpha \in \mathbb{R}$, $\vct{\beta} \in \mathbb{R}^n$ and all $i$. The center row gives $\alpha=\theta_1=0$. Applying the same relation to the two rows of the $j$-th probe pair, whose displacements are $\pm\Delta_0 \vct{e}_j$ and whose temporal offsets are $\theta_{2j}$ and $\theta_{2j+1}$, gives
\[
  \theta_{2j} = \alpha + \Delta_0 \beta_j,
  \qquad
  \theta_{2j+1} = \alpha - \Delta_0 \beta_j
  \qquad\Longrightarrow\qquad
  \alpha = \tfrac{1}{2}\bigl(\theta_{2j} + \theta_{2j+1}\bigr).
\]
Since the probe timestamps follow the center timestamp, $\tfrac12(\theta_{2j}+\theta_{2j+1})>0$, contradicting $\alpha=0$. Hence $\vct{\theta} \notin \operatorname{span}\{\mathbf{1}_{2n+1}, D\}$, and $\operatorname{rank}(X_{\mathrm{aug}}) = n+2$.
\end{proof}

\subsection{Affine-Drift Decoupling in Gradient Recovery}
\label{sec:drift-decoupling}

We next examine how an affine temporal component affects the recovered model parameters.

The recovery operator $\Pi = S^{-1} X_{\mathrm{aug}}^\top K^{-1}$ defined in~\eqref{eq:primal-sol} satisfies
\begin{equation}\label{eq:left-inverse}
  \Pi X_{\mathrm{aug}} \;=\; \bigl(S^{-1} X_{\mathrm{aug}}^\top K^{-1}\bigr) X_{\mathrm{aug}} \;=\; S^{-1} \bigl(X_{\mathrm{aug}}^\top K^{-1} X_{\mathrm{aug}}\bigr) \;=\; S^{-1} S \;=\; I_{n+2}.
\end{equation}
Thus, $\Pi$ is a left inverse of the augmented regressor matrix $X_{\mathrm{aug}}$.

\begin{theorem}
\label{thm:exact-decoupling}
Suppose the observation channel exhibits affine drift:
\begin{equation}\label{eq:affine-channel}
  F(\vct{x}, t) \;=\; f(\vct{x}) + \beta\,t + \varepsilon(\vct{x},t),
\end{equation}
for some unknown drift rate $\beta \in \mathbb{R}$. Let $\vct{F}^{\mathrm{stat}} := \vct{F} - \beta \vct{t} \in \mathbb{R}^m$ denote the underlying static observations that would be obtained in the absence of drift, and let $[c^{\mathrm{stat}}, (\vct{g}^{\mathrm{stat}})^\top, \gamma^{\mathrm{stat}}]^\top := \Pi \vct{F}^{\mathrm{stat}}$ be the model parameters recovered from $\vct{F}^{\mathrm{stat}}$. If $\operatorname{rank}(X_{\mathrm{aug}}) = n+2$, then the parameters recovered by TOBYQA from the drifting vector $\vct{F}$ satisfy:
\begin{equation}\label{eq:decoupling-identity}
  \vct{g} \;=\; \vct{g}^{\mathrm{stat}},
  \qquad
  \gamma \;=\; \gamma^{\mathrm{stat}} + \beta,
  \qquad
  c \;=\; c^{\mathrm{stat}} + \beta\,t_k.
\end{equation}
\end{theorem}

\begin{proof}
Let $\vct{t} = (t_1, \dots, t_m)^\top \in \mathbb{R}^m$ be the vector of query timestamps. Recalling that $t_i = t_k + \theta_i$, we decompose $\vct{t}$ over the columns of $X_{\mathrm{aug}} = [\mathbf{1}_m \;\; D \;\; \vct{\theta}]$:
\begin{equation}\label{eq:t-decomp}
  \vct{t} \;=\; t_k\,\mathbf{1}_m + \vct{\theta} \;=\; X_{\mathrm{aug}} \begin{bmatrix}t_k\\\mathbf{0}_n\\1\end{bmatrix}
  \;=\; X_{\mathrm{aug}}\,(t_k \vct{e}_1 + \vct{e}_{n+2}),
\end{equation}
where $\vct{e}_1 = (1, 0, \dots, 0)^\top \in \mathbb{R}^{n+2}$ and $\vct{e}_{n+2} = (0, \dots, 0, 1)^\top \in \mathbb{R}^{n+2}$.
The observation vector evaluated under~\eqref{eq:affine-channel} can be written as $\vct{F} = \vct{F}^{\mathrm{stat}} + \beta \vct{t}$. Applying the linear recovery operator $\Pi$ to $\vct{F}$ and invoking the left-inverse identity~\eqref{eq:left-inverse}, we obtain:
\begin{align*}
  \begin{bmatrix}c\\\vct{g}\\\gamma\end{bmatrix}
  \;=\; \Pi \vct{F}
  &\;=\; \Pi \vct{F}^{\mathrm{stat}} + \beta\,\Pi \vct{t} \\
  &\;=\; \begin{bmatrix}c^{\mathrm{stat}}\\\vct{g}^{\mathrm{stat}}\\\gamma^{\mathrm{stat}}\end{bmatrix}
  + \beta\,\Pi X_{\mathrm{aug}}\,(t_k \vct{e}_1 + \vct{e}_{n+2}) \\
  &\;=\; \begin{bmatrix}c^{\mathrm{stat}}\\\vct{g}^{\mathrm{stat}}\\\gamma^{\mathrm{stat}}\end{bmatrix}
  + \beta\,(t_k \vct{e}_1 + \vct{e}_{n+2}) \\
  &\;=\; \begin{bmatrix}c^{\mathrm{stat}} + \beta t_k\\\vct{g}^{\mathrm{stat}}\\\gamma^{\mathrm{stat}} + \beta\end{bmatrix}.
\end{align*}
Equating respective components yields~\eqref{eq:decoupling-identity}.
\end{proof}

\begin{remark}
\label{rem:significance}
Theorem~\ref{thm:exact-decoupling} states that, for any drift rate $\beta$, any noise realization, and any sample geometry satisfying the augmented rank condition, adding the affine term $\beta \vct{t}$ does not change the recovered spatial gradient. The result does not depend on the particular Powell kernel used here; it requires only that $K$ be symmetric positive definite, so that $\Pi X_{\mathrm{aug}}=I_{n+2}$. In regression terms, this is an instance of the Frisch--Waugh--Lovell principle~\cite{FrischWaugh1933,Lovell1963}: the temporal column is treated as a nuisance regressor, and the spatial coefficients are recovered after accounting for that column.
\end{remark}

We next generalize the analysis to nonlinear observation drift $\mu(t)$ with bounded second derivative.

\begin{theorem}
\label{thm:nonlinear-drift}
Suppose the observation channel has the form $F(\vct{x},t) = f(\vct{x}) + \mu(t) + \varepsilon(\vct{x},t)$, where $\mu \in C^2(\mathbb{R})$ with $|\mu''(t)| \le L_\mu$ for all $t$. Let $\tau_{\max} := \max_{i} |t_i - t_k|$ denote the maximum temporal span of the active set $\mathcal{Y}_k$. Then:
\begin{equation}\label{eq:nonlinear-bounds}
  \bigl|\gamma - \gamma^{\mathrm{stat}} - \mu'(t_k)\bigr| \;\le\; C_\gamma\,L_\mu\,\tau_{\max},
  \qquad
  \bigl\|\vct{g} - \vct{g}^{\mathrm{stat}}\bigr\|_2 \;\le\; C_g\,L_\mu\,\tau_{\max}^2,
\end{equation}
where, for a fixed normalized sample geometry, the constants $C_\gamma, C_g > 0$ are independent of $\tau_{\max}$ and $L_\mu$; they depend on the normalized spatial and temporal configuration and on $\rho$.
\end{theorem}

\begin{proof}
Performing a second-order Taylor expansion of $\mu(t_i)$ around the center timestamp $t_k$ gives
\[
  \mu(t_i) \;=\; \mu(t_k) + \mu'(t_k)\theta_i + r_\mu(t_i),
  \qquad
  |r_\mu(t_i)| \;\le\; \tfrac{1}{2} L_\mu \theta_i^2 \;\le\; \tfrac{1}{2} L_\mu \tau_{\max}^2.
\]
In vector notation,
\begin{align*}
  \vct{\mu} \;=\; \mu(t_k)\mathbf{1}_m + \mu'(t_k)\vct{\theta} + \vct{r}_\mu
  \;=\; X_{\mathrm{aug}}\bigl[\mu(t_k)\, \vct{e}_1 + \mu'(t_k)\, \vct{e}_{n+2}\bigr] + \vct{r}_\mu.
\end{align*}
Applying the recovery operator $\Pi$, the constant and linear terms in time are represented by the corresponding columns of $X_{\mathrm{aug}}$, leaving the perturbation induced by the second-order residual $\vct{r}_\mu$:
\[
  \begin{bmatrix}c - c^{\mathrm{stat}} - \mu(t_k)\\\vct{g} - \vct{g}^{\mathrm{stat}}\\\gamma - \gamma^{\mathrm{stat}} - \mu'(t_k)\end{bmatrix}
  \;=\; \Pi\,\vct{r}_\mu.
\]
It remains to bound the components of $\Pi\,\vct{r}_\mu$. Write $\tau := \tau_{\max}$ and factor the temporal column as $\vct{\theta} = \tau\,\vct{\hat\theta}$ with $\|\vct{\hat\theta}\|_\infty \le 1$. With
\[
  T := \begin{bmatrix}\mathbf{1}_m & D & \vct{\hat\theta}\end{bmatrix},
  \qquad
  J := \operatorname{diag}(I_{n+1},\, \tau),
  \qquad\text{so that}\qquad
  X_{\mathrm{aug}} = TJ,
\]
we have $S = X_{\mathrm{aug}}^\top K^{-1} X_{\mathrm{aug}} = J\,\hat S\,J$ with $\hat S := T^\top K^{-1} T$. The matrix $K = A + \rho^{-1}I_m$ depends only on the spatial displacements and $\rho$, so both $K$ and $\hat S$ are independent of $\tau$ under this scaling, and
\[
  \Pi\, \vct{r}_\mu
  \;=\; S^{-1} X_{\mathrm{aug}}^\top K^{-1} \vct{r}_\mu
  \;=\; J^{-1} \hat S^{-1} T^\top K^{-1} \vct{r}_\mu,
\]
where $J^{-1} = \operatorname{diag}(I_{n+1}, \tau^{-1})$. The vector $\hat S^{-1} T^\top K^{-1} \vct{r}_\mu$ is independent of $\tau$, so the scaling of the components of $\Pi \vct{r}_\mu$ is governed solely by $J^{-1}$: the $\gamma$-component carries one factor $\tau^{-1}$, while the $c$- and $\vct{g}$-components carry none. Combining this with $\|\vct{r}_\mu\|_2 \le \sqrt{m}\,\|\vct{r}_\mu\|_\infty \le \tfrac{\sqrt{m}}{2} L_\mu \tau^2$ yields
\begin{align*}
  \bigl|\gamma - \gamma^{\mathrm{stat}} - \mu'(t_k)\bigr|
  &\;\le\; \frac{\|\hat S^{-1} T^\top K^{-1}\|_2 \sqrt{m}}{2}\, L_\mu\, \tau, \\
  \bigl\|\vct{g} - \vct{g}^{\mathrm{stat}}\bigr\|_2
  &\;\le\; \frac{\|\hat S^{-1} T^\top K^{-1}\|_2 \sqrt{m}}{2}\, L_\mu\, \tau^2,
\end{align*}
which is~\eqref{eq:nonlinear-bounds} with $C_\gamma = C_g = \tfrac{\sqrt{m}}{2}\,\|\hat S^{-1} T^\top K^{-1}\|_2$, a quantity independent of $\tau$ and $L_\mu$.
\end{proof}

\begin{remark}
\label{rem:fifo-bound}
Theorem~\ref{thm:nonlinear-drift} shows that the temporal span of the active set directly controls the nonlinear-drift contribution to the recovered gradient. The time-priority replacement rule in Section~\ref{sec:arc-decision} discards the oldest non-center sample and therefore keeps the retained observations temporally compact, subject to the protected center point.
\end{remark}

\subsection{Statistical Interpretation of the Powell Residual Kernel}
\label{sec:statistical-foundation}

We next derive a statistical interpretation of the kernel matrix $K = A + \tfrac{1}{\rho}I_m$.

Recall from Section~\ref{sec:arc-decision} that the step computation uses the model with $H_k = 0$, so the surrogate actually consumed by the decision layer is linear in space: $Q_k(\vct{d}, \theta) = c + \vct{g}^\top \vct{d} + \gamma\theta$. A second-order Taylor expansion of $f$ about $\vct{x}_k$ then gives, for each sample,
\begin{equation}\label{eq:residual-decomp}
  F_i - Q_k(\vct{d}_i, \theta_i) \;=\; \tfrac{1}{2} \vct{d}_i^\top H \vct{d}_i + O(\|\vct{d}_i\|_2^3) + \varepsilon_i,
\end{equation}
where $H = \nabla^2 f(\vct{x}_k)$ is the local Hessian omitted from the decision model and $\varepsilon_i \sim \mathcal{N}(0, \sigma_\varepsilon^2)$ is white observation noise. After the zeroth- and first-order terms and the temporal component have been represented by $(c, \vct{g}, \gamma)$, spatial curvature enters the residual quadratically in the displacement. The following proposition gives the covariance of these quadratic terms under a rotation-invariant prior and relates it to Powell's kernel $A$.

\begin{proposition}
\label{prop:kernel-stat}
Let $H \in \mathbb{R}^{n\times n}$ be a random symmetric matrix drawn from a rotation-invariant ensemble with $\mathbb{E}[H] = 0$ and
\begin{equation}\label{eq:goe-prior}
  \mathbb{E}[H_{ab} H_{cd}] \;=\; \tau^2\,(\delta_{ac}\delta_{bd} + \delta_{ad}\delta_{bc}),
\end{equation}
where $\tau^2 > 0$ reflects the curvature uncertainty scale. Then for any displacement vectors $\vct{d}_i, \vct{d}_j \in \mathbb{R}^n$, the covariance between the unmodeled quadratic forms $q_i := \tfrac{1}{2} \vct{d}_i^\top H \vct{d}_i$ and $q_j := \tfrac{1}{2} \vct{d}_j^\top H \vct{d}_j$ is given by
\begin{equation}\label{eq:cov-kernel}
  \operatorname{Cov}(q_i, q_j) \;=\; \tfrac{\tau^2}{2}\,(\vct{d}_i^\top \vct{d}_j)^2 \;=\; \tau^2 A_{ij}.
\end{equation}
\end{proposition}

\begin{proof}
Writing $q_i = \tfrac{1}{2}\sum_{a,b=1}^n H_{ab} d_{i,a} d_{i,b}$, we have $\mathbb{E}[q_i] = 0$. The covariance is:
\begin{align*}
  \operatorname{Cov}(q_i, q_j)
  &\;=\; \mathbb{E}[q_i q_j]
  \;=\; \tfrac{1}{4} \sum_{a,b=1}^n \sum_{c,d=1}^n \mathbb{E}[H_{ab} H_{cd}]\,d_{i,a} d_{i,b} d_{j,c} d_{j,d} \\
  &\;=\; \tfrac{\tau^2}{4} \sum_{a,b,c,d=1}^n (\delta_{ac}\delta_{bd} + \delta_{ad}\delta_{bc})\,d_{i,a} d_{i,b} d_{j,c} d_{j,d} \\
  &\;=\; \tfrac{\tau^2}{4} \Biggl(\sum_{a,b=1}^n d_{i,a} d_{j,a} d_{i,b} d_{j,b} + \sum_{a,b=1}^n d_{i,a} d_{j,b} d_{i,b} d_{j,a}\Biggr) \\
  &\;=\; \tfrac{\tau^2}{4} \Bigl((\vct{d}_i^\top \vct{d}_j)^2 + (\vct{d}_i^\top \vct{d}_j)^2\Bigr)
  \;=\; \tfrac{\tau^2}{2} (\vct{d}_i^\top \vct{d}_j)^2
  \;=\; \tau^2 A_{ij},
\end{align*}
where the last step uses $A_{ij} = \tfrac{1}{2}(\vct{d}_i^\top \vct{d}_j)^2$. Hence the covariance of the omitted quadratic terms is $\tau^2 A$.
\end{proof}

Under Proposition~\ref{prop:kernel-stat}, treating the unmodeled curvature as a Gaussian random effect with covariance $\tau^2 A$ and observation noise as independent Gaussian with covariance $\sigma_\varepsilon^2 I_m$, the total covariance of the residual vector is
\[
  \Sigma_{\mathrm{res}} \;=\; \tau^2 A + \sigma_\varepsilon^2 I_m \;=\; \tau^2 \Bigl(A + \frac{\sigma_\varepsilon^2}{\tau^2} I_m\Bigr) \;=\; \tau^2 \bigl(A + \rho^{-1} I_m\bigr) \;=\; \tau^2 K,
\]
where $\rho = \tau^2 / \sigma_\varepsilon^2$ is the signal-to-noise ratio between curvature variation and observation noise. The generalized least squares (GLS) estimator of $[c, \vct{g}^\top, \gamma]^\top$ under residual covariance $\Sigma_{\mathrm{res}}$ is $(X_{\mathrm{aug}}^\top \Sigma_{\mathrm{res}}^{-1} X_{\mathrm{aug}})^{-1} X_{\mathrm{aug}}^\top \Sigma_{\mathrm{res}}^{-1} \vct{F} = (X_{\mathrm{aug}}^\top K^{-1} X_{\mathrm{aug}})^{-1} X_{\mathrm{aug}}^\top K^{-1} \vct{F} = \Pi \vct{F}$.

Thus, under this residual model, the soft-constrained KKT estimate~\eqref{eq:tobyqa-kkt} coincides with the generalized least-squares estimator.

\subsection{Drift-Compensated Reduction and Error Decomposition}
\label{sec:ared-decomposition}

Finally, we analyze the drift-compensated actual reduction $\mathrm{Ared}_k$ defined in~\eqref{eq:ared-def}.

\begin{proposition}
\label{prop:ared-decomp}
Suppose queries are performed through the channel $F(\vct{x},t) = \phi(f(\vct{x}),t) + \varepsilon(\vct{x},t)$. For a trial step $\vct{s}_k$ evaluated at $t_{\mathrm{new}}$ with center $\vct{x}_k$ at $t_{\mathrm{old}}$, the actual reduction has the decomposition
\begin{equation}\label{eq:ared-three-terms}
  \mathrm{Ared}_k \;=\; \Delta f(\vct{x}_k, \vct{s}_k) \;+\; \mathcal{E}_{\mathrm{drift}} \;+\; \mathcal{E}_{\mathrm{noise}},
\end{equation}
where:
\begin{enumerate}[\rm(i)]
  \item $\Delta f(\vct{x}_k, \vct{s}_k) := f(\vct{x}_k) - f(\vct{x}_k + \vct{s}_k)$ is the latent-objective reduction.
  \item $\mathcal{E}_{\mathrm{drift}} := \gamma_k (t_{\mathrm{new}} - t_{\mathrm{old}}) - \bigl(\phi(f(\vct{x}_k+\vct{s}_k), t_{\mathrm{new}}) - \phi(f(\vct{x}_k), t_{\mathrm{old}})\bigr) - \Delta f(\vct{x}_k, \vct{s}_k)$ is the temporal compensation error.
  \item $\mathcal{E}_{\mathrm{noise}} := \varepsilon(\vct{x}_k, t_{\mathrm{old}}) - \varepsilon(\vct{x}_k + \vct{s}_k, t_{\mathrm{new}})$ is the zero-mean observation noise error with variance $2\sigma_\varepsilon^2$.
\end{enumerate}
In particular, under affine drift $\phi(f(\vct{x}), t) = f(\vct{x}) + \beta t$, Theorem~\ref{thm:exact-decoupling} gives
\begin{equation}\label{eq:ared-affine-residual}
  \mathcal{E}_{\mathrm{drift}}
  \;=\; (\gamma_k-\beta)\,(t_{\mathrm{new}}-t_{\mathrm{old}})
  \;=\; \gamma_k^{\mathrm{stat}}\,(t_{\mathrm{new}}-t_{\mathrm{old}}).
\end{equation}
Thus the contribution of the affine drift rate $\beta$ cancels from $\mathcal{E}_{\mathrm{drift}}$. A residual offset may remain because the drift-free spatial data can induce a nonzero coefficient $\gamma_k^{\mathrm{stat}}$ through space--time correlation in the sequential sample geometry.
\end{proposition}

\begin{proof}
Substituting the channel expressions for $F_{\mathrm{old}}$ and $F_{\mathrm{new}}$ into~\eqref{eq:ared-def} and grouping the spatial, temporal, and noise terms yields~\eqref{eq:ared-three-terms}. Under affine drift, the temporal difference contributes $\beta(t_{\mathrm{new}}-t_{\mathrm{old}})$, while Theorem~\ref{thm:exact-decoupling} gives $\gamma_k=\gamma_k^{\mathrm{stat}}+\beta$. Substitution yields~\eqref{eq:ared-affine-residual}.
\end{proof}

\section{Convergence and Complexity Analysis}
\label{sec:complexity}

In this section, we establish the convergence properties and expected oracle complexity of TOBYQA under probabilistically fully-linear models.

\subsection{Standing Assumptions and Stochastic Framework}
\label{sec:assumptions}

We analyze TOBYQA under the classical framework of probabilistically fully-linear models~\cite{ChenMenickellyScheinberg2018,CartisRoberts2019,BellaviaGurioliMoriniToint2019}. Let $\mathcal{F}_{k-1} := \sigma(\vct{x}_0, \dots, \vct{x}_k, \mathcal{Y}_0, \dots, \mathcal{Y}_k)$ denote the filtration representing the history of the algorithm up to the beginning of iteration $k$. We impose the following standard assumptions on the latent objective function and the observation channel.

\begin{assumption}
\label{ass:smoothness}
The objective function $f: \mathbb{R}^n \to \mathbb{R}$ is continuously differentiable with $L_g$-Lipschitz continuous gradient:
\begin{equation}\label{eq:grad-lipschitz}
  \|\nabla f(\vct{x}) - \nabla f(\vct{y})\|_2 \;\le\; L_g\,\|\vct{x} - \vct{y}\|_2, \qquad \forall\, \vct{x}, \vct{y} \in \mathbb{R}^n,
\end{equation}
and is bounded from below by $f_{\mathrm{low}} > -\infty$ on $\mathbb{R}^n$.
\end{assumption}

\begin{assumption}
\label{ass:radius-bounds}
\textup{\textbf{(Active-set radius bounds.)}}
There exist constants $0<\delta_{\mathrm{lo}}\le\delta_{\mathrm{hi}}<\infty$ such that, at every iteration prior to termination,
\begin{equation}\label{eq:radius-bounds}
  \delta_{\mathrm{lo}}\;\le\;\delta_k\;\le\;\delta_{\mathrm{hi}},
  \qquad
  \delta_k:=\max_{i=1,\dots,m}\|\vct{x}_i-\vct{x}_k\|_2.
\end{equation}
The time-priority update and extrapolation guard used by the implementation do not by themselves enforce these two bounds; they are an explicit regularity assumption for the complexity analysis.
\end{assumption}

\begin{assumption}
\label{ass:fully-linear}
\textup{\textbf{(Probabilistically fully-linear gradient.)}}
There exist constants $\kappa_{eg} > 0$ and failure probability $p \in [0, \tfrac{1}{2})$ such that, for every iteration $k \ge 0$, conditionally on $\mathcal{F}_{k-1}$, the recovered model gradient $\vct{g}_k$ satisfies
\begin{equation}\label{eq:grad-error-bound}
  \mathbb{P}\Bigl(\|\vct{g}_k - \nabla f(\vct{x}_k)\|_2 \;\le\; \kappa_{eg}\,\zeta_k \;\Big|\; \mathcal{F}_{k-1}\Bigr) \;\ge\; 1 - p,
\end{equation}
where $\zeta_k := \max\{\delta_k, \sigma_\varepsilon / \delta_k\}$ and $\delta_k$ is the active sample radius in~\eqref{eq:radius-bounds}.
\end{assumption}

Assumption~\ref{ass:fully-linear} is the standard probabilistic accuracy condition used in the analysis. Its error scale combines the sample radius $\delta_k$, which controls spatial truncation, and the ratio $\sigma_\varepsilon / \delta_k$, which accounts for noise amplification.

For the observation channel, Proposition~\ref{prop:ared-decomp} and Theorem~\ref{thm:exact-decoupling} show that the affine drift rate is represented by the temporal coefficient and does not enter the recovered spatial gradient. They do not imply that the drift-free time coefficient $\gamma_k^{\mathrm{stat}}$ vanishes. Sequential space--time correlation can produce a nonzero $\gamma_k^{\mathrm{stat}}$ even without observation noise. We therefore define the conditional acceptance bias
\begin{equation}\label{eq:ared-expectation}
  b_k \;:=\; \mathbb{E}[\mathrm{Ared}_k \mid \mathcal{F}_{k-1}]
  - \bigl(f(\vct{x}_k)-f(\vct{x}_k+\vct{s}_k)\bigr).
\end{equation}
Under affine drift, relation~\eqref{eq:ared-affine-residual} identifies the deterministic space--time aliasing contribution as $\gamma_k^{\mathrm{stat}}\Delta t_k$; under $C^2$ drift, $b_k$ also contains the nonlinear temporal truncation error. The following assumption controls this conditional bias and the remaining stochastic innovation.

\begin{assumption}
\label{ass:drift-regime}
\textup{\textbf{(Acceptance-channel residual.)}}
There exists $\varepsilon > 0$ such that, whenever the incumbent satisfies $\|\nabla f(\vct{x}_k)\|_2 \ge \varepsilon$, the conditional bias in~\eqref{eq:ared-expectation} obeys
\begin{equation}\label{eq:residual-drift-bound}
  |b_k| \;\le\; \tfrac{1}{4}\,\mathrm{Pred}_k,
\end{equation}
and the centered innovation $\mathrm{Ared}_k-\mathbb{E}[\mathrm{Ared}_k\mid\mathcal F_{k-1}]$ is conditionally sub-Gaussian with variance proxy $2\sigma_\varepsilon^2$.
\end{assumption}

\subsection{Reliable Regularization Regime and Step Decrease}
\label{sec:reliable-regime}

Recall from Section~\ref{sec:arc-decision} that the Zero-Hessian closed-form cubic trial step $\vct{s}_k$ and the corresponding predicted reduction $\mathrm{Pred}_k$ satisfy:
\begin{equation}\label{eq:step-pred-relations}
  \vct{s}_k \;=\; -\frac{\vct{g}_k}{\sqrt{\sigma_k \|\vct{g}_k\|_2}},
  \qquad
  \mathrm{Pred}_k \;=\; \|\vct{g}_k\|_2^{3/2}\,\sigma_k^{-1/2} \;=\; \|\vct{g}_k\|_2\,\|\vct{s}_k\|_2,
\end{equation}
with $\|\vct{s}_k\|_2 = \sqrt{\|\vct{g}_k\|_2 / \sigma_k}$.
We first give a sufficient lower bound on the cubic regularization weight $\sigma_k$ for decrease in the latent objective $f$.

\begin{lemma}
\label{lem:reliable-sigma}
\textup{\textbf{(Descent under sufficient regularization.)}}
Let Assumption~\ref{ass:smoothness} hold. If the cubic regularization weight satisfies
\begin{equation}\label{eq:sigma-threshold}
  \sigma_k \;\ge\; \sigma^*(\vct{g}_k) \;:=\; \frac{4L_g^2}{\|\vct{g}_k\|_2},
\end{equation}
then the objective decrease along the candidate trial step $\vct{s}_k$ satisfies
\begin{equation}\label{eq:true-decrease-bound}
  f(\vct{x}_k) - f(\vct{x}_k + \vct{s}_k) \;\ge\; \tfrac{3}{4}\,\mathrm{Pred}_k \;-\; \|\vct{g}_k - \nabla f(\vct{x}_k)\|_2\,\|\vct{s}_k\|_2.
\end{equation}
\end{lemma}

\begin{proof}
By the classical descent lemma for $L_g$-smooth functions (Assumption~\ref{ass:smoothness}), we have
\begin{align*}
  f(\vct{x}_k + \vct{s}_k) 
  &\;\le\; f(\vct{x}_k) + \nabla f(\vct{x}_k)^\top \vct{s}_k + \tfrac{L_g}{2}\|\vct{s}_k\|_2^2 \\
  &\;=\; f(\vct{x}_k) + \vct{g}_k^\top \vct{s}_k + \tfrac{L_g}{2}\|\vct{s}_k\|_2^2 + \bigl(\nabla f(\vct{x}_k) - \vct{g}_k\bigr)^\top \vct{s}_k.
\end{align*}
Using $\vct{g}_k^\top \vct{s}_k = -\|\vct{g}_k\|_2 \|\vct{s}_k\|_2 = -\mathrm{Pred}_k$ and Cauchy--Schwarz, rearranging terms yields
\begin{equation}\label{eq:descent-intermediate}
  f(\vct{x}_k) - f(\vct{x}_k + \vct{s}_k) \;\ge\; \mathrm{Pred}_k \;-\; \tfrac{L_g}{2}\|\vct{s}_k\|_2^2 \;-\; \|\vct{g}_k - \nabla f(\vct{x}_k)\|_2\,\|\vct{s}_k\|_2.
\end{equation}
We now evaluate the ratio of the quadratic curvature penalty $\tfrac{L_g}{2}\|\vct{s}_k\|_2^2$ to the predicted decrease $\mathrm{Pred}_k$. Substituting the closed-form identities~\eqref{eq:step-pred-relations} gives:
\begin{equation}\label{eq:curvature-ratio}
  \frac{\tfrac{L_g}{2}\|\vct{s}_k\|_2^2}{\mathrm{Pred}_k}
  \;=\;
  \frac{\tfrac{L_g}{2}\bigl(\|\vct{g}_k\|_2 / \sigma_k\bigr)}{\|\vct{g}_k\|_2^{3/2} \sigma_k^{-1/2}}
  \;=\;
  \frac{L_g}{2\sqrt{\sigma_k \|\vct{g}_k\|_2}}.
\end{equation}
Under condition~\eqref{eq:sigma-threshold}, we have $\sigma_k \|\vct{g}_k\|_2 \ge 4L_g^2$, which implies $\sqrt{\sigma_k \|\vct{g}_k\|_2} \ge 2L_g$. Substituting this inequality into~\eqref{eq:curvature-ratio} yields
\[
  \frac{\tfrac{L_g}{2}\|\vct{s}_k\|_2^2}{\mathrm{Pred}_k} \;\le\; \frac{L_g}{2(2L_g)} \;=\; \tfrac{1}{4},
  \qquad\text{whence}\qquad
  \tfrac{L_g}{2}\|\vct{s}_k\|_2^2 \;\le\; \tfrac{1}{4}\,\mathrm{Pred}_k.
\]
Substituting $\tfrac{L_g}{2}\|\vct{s}_k\|_2^2 \le \tfrac{1}{4}\mathrm{Pred}_k$ into~\eqref{eq:descent-intermediate} yields~\eqref{eq:true-decrease-bound}, completing the proof.
\end{proof}

\begin{remark}
\label{rem:feedback-control}
\textup{\textbf{(Interpretation of the regularization scale.)}}
Condition~\eqref{eq:sigma-threshold} identifies the scale at which the curvature term in the descent bound is at most one quarter of the predicted reduction. In Algorithm~\ref{alg:tobyqa}, changing $\sigma_k$ changes the step length through $\|\vct{s}_k\|_2=\sqrt{\|\vct{g}_k\|_2/\sigma_k}$: rejection increases $\sigma_k$, whereas an accepted step with $\rho_k>\eta$ decreases it. Thus $\sigma_k$ controls the step length without a separately maintained trust-region radius.
\end{remark}

\subsection{Rejection Probability and Regularization Confinement}
\label{sec:supermartingale}

\begin{lemma}
\label{lem:reject-prob}
\textup{\textbf{(Rejection probability in the reliable regime.)}}
Let Assumption~\ref{ass:smoothness} hold, let the residual drift bound~\eqref{eq:residual-drift-bound} hold at iteration $k$, and suppose that
\begin{equation}\label{eq:reliable-conditions}
  \sigma_k \;\ge\; \gamma_2\,\sigma^*(\vct{g}_k)
  \qquad\text{and}\qquad
  \|\vct{g}_k - \nabla f(\vct{x}_k)\|_2\,\|\vct{s}_k\|_2 \;\le\; \tfrac{1}{4}\,\mathrm{Pred}_k.
\end{equation}
Then the conditional rejection probability obeys
\begin{equation}\label{eq:reject-prob}
  \mathbb{P}\bigl(\text{iteration } k \text{ is rejected} \,\big|\, \mathcal{F}_{k-1}\bigr)
  \;\le\;
  e^{-\kappa_{\mathrm{tol}}^2/8}
  \;=:\; p',
\end{equation}
and $p' \le \tfrac{1}{2}$ whenever $\kappa_{\mathrm{tol}} \ge \sqrt{8\ln 2}$.
\end{lemma}

\begin{proof}
Under $\sigma_k \ge \gamma_2\sigma^*(\vct{g}_k)$ we have $\sigma_k\|\vct{g}_k\|_2 \ge 4L_g^2$, so Lemma~\ref{lem:reliable-sigma} gives $\Delta f_k \ge \tfrac{3}{4}\mathrm{Pred}_k - \|\vct{g}_k - \nabla f(\vct{x}_k)\|_2\|\vct{s}_k\|_2 \ge \tfrac{1}{2}\mathrm{Pred}_k$. Together with~\eqref{eq:residual-drift-bound} and the expectation~\eqref{eq:ared-expectation}, this yields
\[
  \mathbb{E}[\mathrm{Ared}_k \mid \mathcal{F}_{k-1}]
  \;=\; \Delta f_k + b_k
  \;\ge\; \tfrac{1}{2}\mathrm{Pred}_k - \tfrac{1}{4}\mathrm{Pred}_k
  \;=\; \tfrac{1}{4}\mathrm{Pred}_k \;>\; 0.
\]
The deviation $\mathrm{Ared}_k - \mathbb{E}[\mathrm{Ared}_k \mid \mathcal{F}_{k-1}]$ is sub-Gaussian with variance proxy $2\sigma_\varepsilon^2 = 2R_{\mathrm{obs}}$. Hence, by the sub-Gaussian tail bound and the noise gate~\eqref{eq:noise-gate},
\begin{align*}
  \mathbb{P}\bigl(\mathrm{Ared}_k < -\tfrac{\kappa_{\mathrm{tol}}}{2}\sqrt{2R_{\mathrm{obs}}} \,\big|\, \mathcal{F}_{k-1}\bigr)
  &\;\le\;
  \exp\Bigl(-\frac{\bigl(\tfrac{\kappa_{\mathrm{tol}}}{2}\sqrt{2R_{\mathrm{obs}}} + \tfrac{1}{4}\mathrm{Pred}_k\bigr)^2}{4R_{\mathrm{obs}}}\Bigr) \\
  &\;\le\;
  \exp\Bigl(-\frac{\tfrac{\kappa_{\mathrm{tol}}^2}{4}\,2R_{\mathrm{obs}}}{4R_{\mathrm{obs}}}\Bigr)
  \;=\; e^{-\kappa_{\mathrm{tol}}^2/8},
\end{align*}
where the last inequality uses $\mathrm{Pred}_k > 0$, and $p' \le \tfrac{1}{2}$ for $\kappa_{\mathrm{tol}} \ge \sqrt{8\ln 2}$ since $e^{-\ln 2} = \tfrac{1}{2}$.
\end{proof}

Lemma~\ref{lem:reject-prob} bounds the rejection probability in the reliable regime. An almost-sure upper bound on $\sigma_k$ would additionally require a uniform lower bound on the probability of an acceptance with $\rho_k>\eta$, which is not included in Assumptions~\ref{ass:smoothness}--\ref{ass:drift-regime}. We therefore impose regularization confinement as a standing assumption.

\begin{assumption}
\label{ass:sigma-confinement}
\textup{\textbf{(Regularization confinement.)}}
For the target accuracy $\varepsilon > 0$ with $\kappa_{\mathrm{tol}} \ge \sqrt{8\ln 2}$, the regularization parameter satisfies
\begin{equation}\label{eq:sigma-max}
  \sigma_k \;\le\; \sigma_{\max} \;:=\; \gamma_2\,\max\Biggl\{\sigma_0,\; \frac{8L_g^2}{\varepsilon}\Biggr\},
  \qquad \text{almost surely,}
\end{equation}
throughout all iterations $k$ prior to termination, i.e., while $\min_{j \le k} \|\nabla f(\vct{x}_j)\|_2 \ge \varepsilon$.
\end{assumption}

\begin{proposition}
\label{prop:rejected-bound}
\textup{\textbf{(Bound on unsuccessful iterations.)}}
Let $\mathcal{S}_K := \{k \le K : \text{iteration } k \text{ is accepted}\}$ and $\mathcal{U}_K := \{k \le K : \text{iteration } k \text{ is rejected}\}$ denote the index sets of successful and unsuccessful iterations up to iteration $K$, respectively. Then:
\begin{equation}\label{eq:rejected-count}
  |\mathcal{U}_K| \;\le\; \frac{\log \gamma_1}{\log \gamma_2}\,|\mathcal{S}_K| \;+\; \frac{1}{\log \gamma_2}\log\frac{\sigma_{\max}}{\sigma_0}.
\end{equation}
\end{proposition}

\begin{proof}
From the update rule in Algorithm~\ref{alg:tobyqa}, the weight decreases by $\gamma_1$ only on accepted iterations with $\rho_k > \eta$, is held fixed on accepted iterations with $\rho_k \le \eta$, and increases by $\gamma_2$ on rejected iterations. Let $N_1 \le |\mathcal{S}_K|$ denote the number of accepted iterations on which $\rho_k > \eta$. Taking logarithms and telescoping from $k = 0$ to $K-1$ gives
\[
  \log \sigma_K \;=\; \log \sigma_0 \;-\; N_1\,\log \gamma_1 \;+\; |\mathcal{U}_K|\,\log \gamma_2.
\]
Since $N_1 \le |\mathcal{S}_K|$ and $\sigma_K \le \sigma_{\max}$ by Assumption~\ref{ass:sigma-confinement}, rearranging yields
\[
  |\mathcal{U}_K|\,\log \gamma_2
  \;=\; \log\frac{\sigma_K}{\sigma_0} + N_1 \log \gamma_1
  \;\le\; \log\frac{\sigma_{\max}}{\sigma_0} + |\mathcal{S}_K|\,\log \gamma_1,
\]
and dividing by $\log \gamma_2 > 0$ establishes~\eqref{eq:rejected-count}.
\end{proof}

\subsection{Expected Iteration and Oracle Complexity}
\label{sec:expected-complexity}

We now state the iteration and oracle complexity bound for TOBYQA.

\begin{theorem}
\label{thm:main-complexity}
\textup{\textbf{(Expected oracle call complexity.)}}
Let Assumptions~\ref{ass:smoothness}--\ref{ass:sigma-confinement} hold. Let $\varepsilon > 0$ be any target gradient accuracy exceeding the statistical noise resolution floor
\begin{equation}\label{eq:noise-floor}
  \varepsilon \;\ge\; \varepsilon_{\mathrm{noise}} \;:=\; C_0\,\kappa_{eg}\,\max\Bigl\{\delta_{\mathrm{hi}},\; \frac{\sigma_\varepsilon}{\delta_{\mathrm{lo}}}\Bigr\},
\end{equation}
where $C_0 > 0$ is an absolute constant. Then the total number of iterations $K_\varepsilon$ (and consequently the number of oracle evaluations after initialization) required to achieve $\|\nabla f(\vct{x}_k)\|_2 \le \varepsilon$ satisfies
\begin{equation}\label{eq:complexity-rate}
  \mathbb{E}[K_\varepsilon] \;=\; O\Biggl(\frac{L_g\,\bigl(f(\vct{x}_0) - f_{\mathrm{low}}\bigr)}{\varepsilon^2}\Biggr).
\end{equation}
\end{theorem}

The restriction $\varepsilon\ge\varepsilon_{\mathrm{noise}}$ reflects the statistical resolution of single noisy queries: without replication, reductions below the observation-noise scale cannot be reliably distinguished from random fluctuations. The posterior-covariance stopping rule in Algorithm~\ref{alg:tobyqa} is calibrated to this same scale.

\begin{proof}
Let $k \in \mathcal{S}_K$ be an accepted iteration prior to termination. Conditioned on the fully-linear event $\mathcal{E}_{k-1}^{\mathrm{FL}}$ from Assumption~\ref{ass:fully-linear}, the model gradient satisfies $\|\vct{g}_k\|_2 \ge \tfrac{1}{2}\|\nabla f(\vct{x}_k)\|_2 \ge \tfrac{1}{2}\varepsilon$. Moreover, since $\varepsilon \ge C_0\kappa_{eg}\zeta_k$ by~\eqref{eq:noise-floor}, the model-error contribution satisfies
\begin{align*}
  \|\vct{g}_k - \nabla f(\vct{x}_k)\|_2\,\|\vct{s}_k\|_2
  &\;\le\; \kappa_{eg}\,\zeta_k\,\|\vct{s}_k\|_2
  \;\le\; \frac{\varepsilon}{C_0}\,\|\vct{s}_k\|_2
  \;\le\; \frac{2\|\vct{g}_k\|_2}{C_0}\,\|\vct{s}_k\|_2 \\
  &\;=\; \frac{2}{C_0}\,\mathrm{Pred}_k
  \;\le\; \tfrac{1}{4}\,\mathrm{Pred}_k,
\end{align*}
provided $C_0 \ge 8$. Combining Lemma~\ref{lem:reliable-sigma}, this gradient-accuracy bound, and Assumption~\ref{ass:drift-regime}, the latent-objective decrease on an accepted step satisfies
\begin{align*}
  f(\vct{x}_k) - f(\vct{x}_{k+1})
  &\;\ge\; \tfrac{1}{4}\,\mathrm{Pred}_k \\
  &\;=\; \tfrac{1}{4}\,\|\vct{g}_k\|_2^{3/2}\,\sigma_k^{-1/2} \\
  &\;\ge\; \tfrac{1}{4}\,\Bigl(\frac{\varepsilon}{2}\Bigr)^{3/2} \sigma_{\max}^{-1/2}.
\end{align*}
Substituting $\sigma_{\max} = \gamma_2\,\max\{\sigma_0, 8L_g^2/\varepsilon\} = \Theta(L_g^2/\varepsilon)$ from~\eqref{eq:sigma-max} yields
\begin{equation}\label{eq:per-step-descent}
  f(\vct{x}_k) - f(\vct{x}_{k+1})
  \;\ge\;
  \frac{1}{4}\,\frac{\varepsilon^{3/2}}{2\sqrt{2}}\,\sqrt{\frac{\varepsilon}{8\gamma_2 L_g^2}}
  \;\ge\;
  c\,\frac{\varepsilon^2}{L_g}
  \;=\;
  \Omega\Biggl(\frac{\varepsilon^2}{L_g}\Biggr),
\end{equation}
with $c > 0$ an absolute constant (e.g.\ $c = (32\sqrt{\gamma_2})^{-1}$ when $\sigma_0 \le 8L_g^2/\varepsilon$).
Summing the descent inequality~\eqref{eq:per-step-descent} over all accepted iterations $k \in \mathcal{S}_K$ gives:
\[
  f(\vct{x}_0) - f_{\mathrm{low}} \;\ge\; f(\vct{x}_0) - f(\vct{x}_K) \;=\; \sum_{k \in \mathcal{S}_K} \bigl(f(\vct{x}_k) - f(\vct{x}_{k+1})\bigr) \;\ge\; |\mathcal{S}_K|\,\cdot\,\Omega\Biggl(\frac{\varepsilon^2}{L_g}\Biggr),
\]
which establishes that the total number of accepted iterations satisfies
\begin{equation}\label{eq:accepted-bound}
  |\mathcal{S}_K| \;\le\; O\Biggl(\frac{L_g\,\bigl(f(\vct{x}_0) - f_{\mathrm{low}}\bigr)}{\varepsilon^2}\Biggr).
\end{equation}
By Proposition~\ref{prop:rejected-bound}, the number of unsuccessful iterations is linearly bounded by $|\mathcal{S}_K|$ up to an additive constant:
\[
  K_\varepsilon \;=\; |\mathcal{S}_K| + |\mathcal{U}_K| \;\le\; \Bigl(1 + \frac{\log \gamma_1}{\log \gamma_2}\Bigr)\,|\mathcal{S}_K| \;+\; \frac{1}{\log \gamma_2}\log\frac{\sigma_{\max}}{\sigma_0}.
\]
Finally, we pass to expectations. The descent bound above holds on the fully-linear event, which occurs with conditional probability at least $1 - p$ at every iteration by Assumption~\ref{ass:fully-linear}; accepted iterations lacking the event contribute no guaranteed descent. Since each iteration is fully linear with conditional probability at least $1 - p$, the expected number of iterations needed to accumulate the $|\mathcal{S}_K|$ descent-producing accepted iterations is inflated by at most the factor $(1 - p)^{-1} \le 2$, which is absorbed into the big-$O$ constant. Since TOBYQA evaluates one new query per iteration after the initial $2n+1$ evaluations, the total number of oracle calls satisfies $\mathbb{E}[N_{\mathrm{eval}}] = 2n + 1 + \mathbb{E}[K_\varepsilon] = O(L_g(f(\vct{x}_0) - f_{\mathrm{low}})\varepsilon^{-2})$, completing the proof.
\end{proof}

\begin{remark}
\label{rem:optimality}
\textup{\textbf{(Comparison with first-order complexity.)}}
The rate $O(\varepsilon^{-2})$ in Theorem~\ref{thm:main-complexity} matches the information-theoretic lower bound for first-order methods on $C^{1,1}$ nonconvex objectives~\cite{CarmonDuchiHinderSidford2020}. This is consistent with the information available to TOBYQA at the budget $m=2n+1$: the active set supports a fully-linear model, but not an asymptotically accurate Hessian. The $O(\varepsilon^{-3/2})$ rate of derivative-based adaptive cubic regularization~\cite{NesterovPolyak2006,CartisGouldToint2011ARC,CartisGouldToint2011ARC2} requires $\|(H_k-\nabla^2f(\vct{x}_k))\vct{s}_k\|_2=O(\|\vct{s}_k\|_2^2)$~\cite[AM.4]{CartisGouldToint2011ARC}, whose recovery would generally require $O(n^2)$ samples. Under temporal drift, such a sample set also spans $O(n^2)$ time steps and incurs the nonlinear drift error quantified in Theorem~\ref{thm:nonlinear-drift}. TOBYQA instead uses an $O(n)$ active set, with a shorter temporal span and the first-order complexity rate.
\end{remark}

\section{Numerical Experiments}
\label{sec:experiments}

\noindent\textbf{Experimental Setup.}
The benchmark testbed comprises $65$ scalable smooth unconstrained test problems in the CUTEst style~\cite{MoreWild2009}, evaluated at dimensions $n \in \{6, 10, 20\}$ with $5$ random seeds controlling the observation noise and drift realizations. Each problem is queried through one of the seven evaluation channels summarized in Table~\ref{tab:envs}. Channel amplitudes and observation noise variances are scaled proportionally to the initial function magnitude $s_0 = \max(|f(\vct{x}_0)|, 1)$, with relative noise scale $\sigma_\varepsilon = 10^{-3} s_0$. The temporal query index advances as $t \leftarrow t + 1$ on each oracle evaluation across all solvers.

\begin{table}[t]
  \centering\small
  \caption{Observation drift channels used in the benchmark experiments. Here $s_0 = \max(|f(\vct{x}_0)|, 1)$, $t$ is the sequential oracle query index, and the Lorenz trajectory is generated with parameters $(\sigma, \rho, \beta) = (10, 28, 8/3)$, initial state $(1,1,1)$, and integration step $0.01$.}
  \label{tab:envs}
  \begin{tabular}{lll}
    \toprule
    Channel & Analytical Form $\phi(f(\vct{x}), t)$ & Classification \\
    \midrule
    Static        & $f(\vct{x})$                                      & Baseline ($\beta = 0$) \\
    LinearDrift   & $f(\vct{x}) + 0.01\,s_0\,t$                       & Affine drift (Theorem~\ref{thm:exact-decoupling}) \\
    PeriodicAdd   & $f(\vct{x}) + 0.05\,s_0\sin(2\pi t/50)$           & $C^2$ bounded drift (Theorem~\ref{thm:nonlinear-drift}) \\
    ChaoticDrift  & $f(\vct{x}) + 0.03\,s_0\,\mathrm{Lorenz}_x(t)$    & Chaotic dynamic drift \\
    Hetero        & $f(\vct{x}) + 0.01\,s_0\,t\,(1+\tanh\|\vct{x}-\vct{x}_0\|)$   & Spatially coupled drift \\
    MultCoupling  & $f(\vct{x})\cdot(1 + 10^{-4}t)$                   & Multiplicative drift \\
    PeriodicMult  & $f(\vct{x})\cdot(1 + 0.05\sin(2\pi t/50))$        & Multiplicative oscillatory drift \\
    \bottomrule
  \end{tabular}
\end{table}

We benchmark TOBYQA against representative classical and modern derivative-free solvers:
\begin{enumerate}[\rm(i)]
  \item \textbf{NEWUOA \& BOBYQA}~\cite{Powell2006NEWUOA,Powell2009BOBYQA}: Powell's classical quadratic interpolation solvers, executed via the authoritative PDFO library~\cite{RagonneauZhang2024};
  \item \textbf{DFO-LS}~\cite{CartisRoberts2019}: A modern model-based derivative-free solver with noise-aware regression enabled;
  \item \textbf{Nelder--Mead}~\cite{NelderMead1965}: The classical heuristic simplex search method;
  \item \textbf{BFGS (FD)}: Quasi-Newton BFGS using forward-difference gradient approximations ($h = 10^{-4}$).
\end{enumerate}
Every solver is allocated a maximum budget of $1500$ evaluations per instance. Performance is measured by evaluating every queried point on the noise-free static latent objective $f$. Following Mor\'e and Wild~\cite{MoreWild2009}, an instance is declared solved at accuracy tolerance $\tau \in \{10^{-1}, 10^{-3}, 10^{-5}, 10^{-7}\}$ if $\min_k f(\vct{x}_k) \le f_L + \tau(f(\vct{x}_0) - f_L)$, where $f_L$ is the global minimum achieved across all solvers.

\subsection{Benchmark Comparison with Existing Solvers}
\label{sec:bench-comparison}

\begin{table}[t]
\centering\footnotesize
\setlength{\tabcolsep}{3pt}
\caption{Solve rates (\%). Budget: 1500 evaluations per instance.}
\label{tab:solve-rates}
\begin{tabular}{lrrrrrrrrrrrr}
\toprule
Solver & \multicolumn{3}{c}{$\tau=10^{-1}$} & \multicolumn{3}{c}{$\tau=10^{-3}$} & \multicolumn{3}{c}{$\tau=10^{-5}$} & \multicolumn{3}{c}{$\tau=10^{-7}$} \\
\cmidrule(lr){2-4}\cmidrule(lr){5-7}\cmidrule(lr){8-10}\cmidrule(lr){11-13}
 & $n{=}6$ & $n{=}10$ & $n{=}20$ & $n{=}6$ & $n{=}10$ & $n{=}20$ & $n{=}6$ & $n{=}10$ & $n{=}20$ & $n{=}6$ & $n{=}10$ & $n{=}20$ \\
\midrule
\textbf{TOBYQA} & 91.7 & 88.7 & 87.4 & 71.0 & 60.8 & 41.9 & 37.6 & 33.2 & 26.4 & 31.4 & 29.0 & 24.3 \\
NEWUOA & 71.6 & 61.1 & 56.5 & 20.7 & 18.5 & 14.9 & 8.5 & 11.3 & 10.4 & 5.2 & 7.1 & 8.7 \\
BOBYQA & 61.9 & 58.2 & 51.6 & 18.6 & 16.1 & 10.8 & 7.7 & 10.2 & 8.4 & 4.2 & 6.7 & 6.6 \\
DFO-LS & 70.2 & 58.7 & 38.9 & 25.1 & 8.8 & 4.9 & 8.4 & 4.8 & 3.8 & 5.3 & 3.2 & 3.3 \\
Nelder--Mead & 34.5 & 20.0 & 11.9 & 8.6 & 3.7 & 3.2 & 3.7 & 2.2 & 2.9 & 2.0 & 1.1 & 1.4 \\
BFGS (FD) & 14.0 & 10.4 & 1.3 & 1.2 & 0.7 & 0.7 & 0.7 & 0.7 & 0.7 & 0.7 & 0.7 & 0.7 \\
\bottomrule
\end{tabular}
\end{table}

Table~\ref{tab:solve-rates} reports the overall solve rates across dimensions $n \in \{6, 10, 20\}$ and target tolerances $\tau \in \{10^{-1}, 10^{-3}, 10^{-5}, 10^{-7}\}$, along with the median number of evaluations among solved instances. Figure~\ref{fig:data-profile} displays the cumulative data profiles at tolerance $\tau = 10^{-3}$, and Figure~\ref{fig:rate-vs-tau} depicts the solve rate decay curves across tolerances spanning seven orders of magnitude.

\begin{figure}[t]
  \centering
  \includegraphics[width=.72\textwidth]{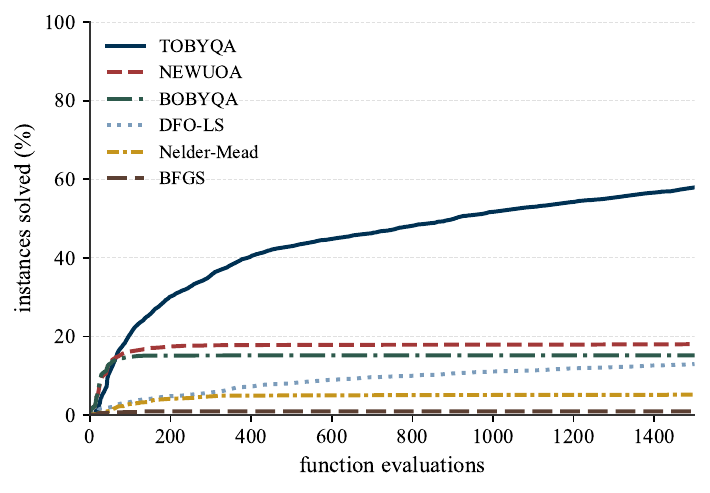}
  \caption{Data profiles at tolerance $\tau = 10^{-3}$: fraction of benchmark instances solved as a function of the number of function evaluations, pooled across all test problems, dimensions, and drift channels.}
  \label{fig:data-profile}
\end{figure}

\begin{figure}[t]
  \centering
  \includegraphics[width=.72\textwidth]{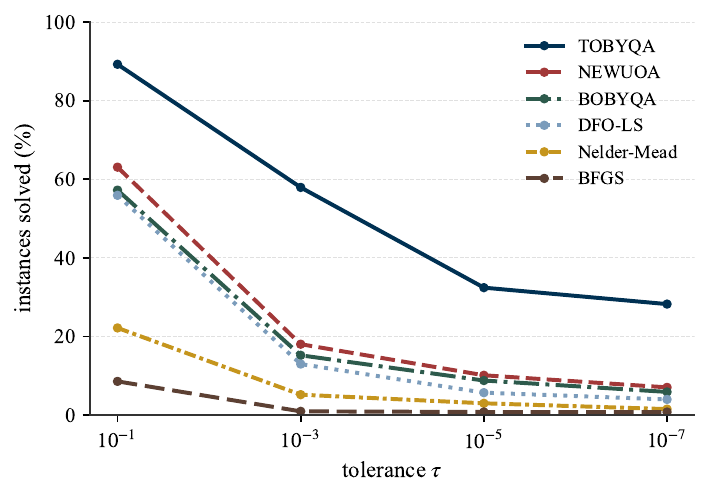}
  \caption{Solve rates as a function of convergence tolerance $\tau \in [10^{-7}, 10^{-1}]$, pooled across all benchmark configurations.}
  \label{fig:rate-vs-tau}
\end{figure}

The numerical results show the following patterns:
\begin{enumerate}[\rm(i)]
  \item \textbf{Solve rates across tolerances:} TOBYQA has the highest solve rate at every reported tolerance and dimension. At $\tau = 10^{-3}$, TOBYQA solves $71.0\%$ of instances at $n = 6$, $60.8\%$ at $n = 10$, and $41.9\%$ at $n = 20$, compared with $20.7\%/18.5\%/14.9\%$ for NEWUOA and $18.6\%/16.1\%/10.8\%$ for BOBYQA.
  \item \textbf{Static interpolation under drift:} NEWUOA and BOBYQA attain solve rates of $26.6\%$ and $22.1\%$ in the static environment (Table~\ref{tab:env-breakdown}), while their rates under additive temporal drift are approximately $9\%$--$15\%$. In these methods, temporal variation enters the fitted spatial quadratic model and can alter the search direction.
  \item \textbf{Finite-difference BFGS under drift:} BFGS (FD) has low solve rates under the drifting, noisy channels. With $\sigma_\varepsilon = 10^{-3} s_0$ and finite-difference step $h = 10^{-4}$, the difference quotients amplify observation noise and incur an $O(\beta/h)$ bias under affine drift.
  \item \textbf{Evaluation-budget use:} As shown in Table~\ref{tab:budget}, NEWUOA and BOBYQA return before exhausting the budget on $99\%\sim100\%$ of runs, with median totals of $147$--$185$ evaluations out of $1500$. TOBYQA reaches the evaluation budget in most runs and otherwise terminates through the posterior-covariance criterion~\eqref{eq:stopping-rule}.
\end{enumerate}

\begin{table}[t]
\centering\footnotesize
\setlength{\tabcolsep}{3pt}
\caption{Solve rates (\%) at $\tau=10^{-3}$ broken down by drift regime, pooled over all $n$ and seeds.}
\label{tab:env-breakdown}
\begin{tabular}{lrrrrrrr}
\toprule
Solver & Chaotic & Hetero & Linear & MultCoupling & PeriodicAdd & PeriodicMult & Static \\
\midrule
\textbf{TOBYQA} & 56.6 & 50.1 & 65.0 & 57.6 & 57.2 & 59.7 & 59.3 \\
NEWUOA & 10.8 & 12.0 & 10.2 & 30.4 & 15.2 & 21.0 & 26.6 \\
BOBYQA & 12.4 & 9.4 & 9.9 & 21.9 & 12.6 & 17.8 & 22.1 \\
DFO-LS & 13.1 & 2.6 & 2.7 & 22.1 & 8.8 & 18.7 & 22.8 \\
Nelder--Mead & 1.5 & 0.0 & 0.6 & 12.8 & 3.1 & 6.6 & 11.5 \\
BFGS (FD) & 1.9 & 0.6 & 1.5 & 0.3 & 0.8 & 0.8 & 0.3 \\
\bottomrule
\end{tabular}
\end{table}

\begin{table}[t]\centering\small
\caption{Budget utilization. An early return indicates that the solver stopped through an internal termination criterion before exhausting the evaluation budget.}
\label{tab:budget}
\begin{tabular}{lrrr}\toprule
Solver & median evals used & \% of budget & \% runs stopping early \\\midrule
\textbf{TOBYQA} & 1500 & 100.0 & 8.5 \\
NEWUOA & 185 & 12.3 & 99.0 \\
BOBYQA & 147 & 9.8 & 99.8 \\
DFO-LS & 1500 & 100.0 & 25.0 \\
Nelder--Mead & 1500 & 100.0 & 0.0 \\
BFGS (FD) & 180 & 12.0 & 99.6 \\
\bottomrule\end{tabular}\end{table}

\subsection{Ablation Studies and Curvature Treatment}
\label{sec:ablations}

Table~\ref{tab:ablation} reports an ablation study across the seven drift channels.

\begin{table}[t]
\centering\footnotesize
\setlength{\tabcolsep}{2.2pt}
\caption{Ablation of TOBYQA components: solve rates (\%) at $\tau=10^{-3}$ by drift regime, pooled over all $n$ and seeds.}
\label{tab:ablation}
\begin{tabular}{lrrrrrrr}
\toprule
Solver & Chaotic & Hetero & Linear & MultCoupling & PeriodicAdd & PeriodicMult & Static \\
\midrule
\textbf{TOBYQA} & 56.6 & 50.1 & 65.0 & 57.6 & 57.2 & 59.7 & 59.3 \\
\quad + interpolated $H$ & 48.4 & 55.3 & 66.5 & 62.7 & 51.5 & 65.1 & 61.2 \\
\quad -- drift compensation & 9.7 & 7.3 & 8.5 & 57.9 & 12.8 & 47.2 & 59.7 \\
\quad -- curvature kernel $A$ & 29.1 & 30.2 & 40.1 & 37.7 & 35.0 & 38.3 & 36.2 \\
\quad + Powell update & 12.2 & 19.0 & 16.9 & 14.5 & 13.9 & 14.8 & 15.1 \\
\quad frozen $\sigma$ & 39.1 & 36.4 & 38.4 & 37.9 & 38.5 & 36.3 & 36.8 \\
\bottomrule
\end{tabular}
\end{table}

\noindent\textbf{Impact of Drift Compensation.}
In this ablation, the time-augmented KKT system is retained for spatial-gradient recovery, but its recovered coefficient is set to $\gamma=0$ in the acceptance statistic~\eqref{eq:ared-def}. The ablated variant is nearly unchanged on the Static channel ($59.7\%$ vs $59.3\%$), whereas its solve rate falls to $7.3\%$--$12.8\%$ across the additive drifting channels. The comparison isolates the practical contribution of drift compensation in the acceptance test.

\noindent\textbf{Role of the Curvature Prior Kernel $A$.}
Replacing the composite kernel $K = A + \rho^{-1} I_m$ with an unweighted ridge penalty $K = \rho^{-1} I_m$ (which treats residuals as independent white noise and ignores curvature correlation) reduces the solve rate from $59.3\%$ to $36.2\%$ on Static and from $65.0\%$ to $40.1\%$ on LinearDrift. This result is consistent with the statistical interpretation in Proposition~\ref{prop:kernel-stat}: the kernel $A_{ij} = \tfrac{1}{2}(\vct{d}_i^\top \vct{d}_j)^2$ accounts for correlated residual variation induced by unmodeled curvature.

\noindent\textbf{Curvature Treatment: Zero or Interpolated Hessian.}
The Zero-Hessian default of Section~\ref{sec:arc-decision} was adopted on identifiability grounds: at $m = 2n+1$ the $\tfrac{1}{2}n(n+1)$ quadratic coefficients are underdetermined. The benchmark data show, however, that the relative merit of the interpolated Hessian $H_k = \sum_i \lambda_i \vct{d}_i \vct{d}_i^\top$ varies with dimension in a way that this argument alone does not predict. Table~\ref{tab:hessian} isolates the two variants. At $n = 6$ the default is better at both tolerances ($71.0\%$ vs $63.6\%$ at $\tau = 10^{-3}$; $37.6\%$ vs $26.4\%$ at $\tau = 10^{-5}$), and the gap widens at the reduced budget of $450$ evaluations ($83.3\%$ vs $67.3\%$). At $n = 20$ the ordering reverses at every tolerance ($51.3\%$ vs $41.9\%$ at $\tau = 10^{-3}$; $39.4\%$ vs $26.4\%$ at $\tau = 10^{-5}$), and at the enlarged budget of $5000$ evaluations the interpolated variant also wins the paired comparison ($1229$ of $2275$ instances against $963$ for the default, $83$ ties). Because the two budget-shifted cells align the per-variable evaluation budget across dimensions, the reversal follows the dimension rather than the budget.

\begin{table}[t]
\centering\footnotesize
\setlength{\tabcolsep}{3pt}
\caption{Solve rates (\%) of the Zero-Hessian default versus the interpolated Hessian, by dimension and tolerance. The first three column groups use the main budget of $1500$ evaluations; the last two use budget-scaled cells ($n=6$ with $450$ evaluations, $n=20$ with $5000$). Pooled over all $65$ problems, $7$ drift channels, and $5$ seeds.}
\label{tab:hessian}
\begin{tabular}{l cc cc cc cc cc}
\toprule
& \multicolumn{2}{c}{$n=6$} & \multicolumn{2}{c}{$n=10$} & \multicolumn{2}{c}{$n=20$} & \multicolumn{2}{c}{$n=6$, b=450} & \multicolumn{2}{c}{$n=20$, b=5000} \\
\cmidrule(lr){2-3}\cmidrule(lr){4-5}\cmidrule(lr){6-7}\cmidrule(lr){8-9}\cmidrule(lr){10-11}
Solver & $10^{-3}$ & $10^{-5}$ & $10^{-3}$ & $10^{-5}$ & $10^{-3}$ & $10^{-5}$ & $10^{-3}$ & $10^{-5}$ & $10^{-3}$ & $10^{-5}$ \\
\midrule
TOBYQA ($H_k \equiv 0$) & 71.0 & 37.6 & 60.8 & 33.2 & 41.9 & 26.4 & 83.3 & 64.3 & 67.1 & 45.9 \\
\quad + interpolated $H$ & 63.6 & 26.4 & 61.2 & 31.6 & 51.3 & 39.4 & 67.3 & 44.9 & 70.8 & 59.7 \\
\bottomrule
\end{tabular}
\end{table}

The environment breakdown at $n = 20$ and budget $5000$ shows a drift-specific pattern: the interpolated Hessian has higher solve rates on Static, Hetero, and LinearDrift ($84.4\%$ vs $60.8\%$ on Static), but lower rates on ChaoticDrift and PeriodicAdd ($39.0\%$ vs $67.2\%$ and $38.2\%$ vs $68.0\%$). The latter two channels contain nonlinear additive drift, which contributes to the recovery error bounded in Theorem~\ref{thm:nonlinear-drift}. The identifiability ratio $(2n+1)/(\tfrac{1}{2}n(n+1))$ decreases monotonically in $n$, yet the interpolated Hessian becomes more useful as $n$ grows, so underdetermination alone does not explain the reversal. We therefore retain the Zero-Hessian default and leave the dimension-dependent role of interpolated curvature as an open question (Section~\ref{sec:conclusion}). Powell's incremental update has solve rates of $12\%\sim 19\%$ across all channels (Table~\ref{tab:ablation}), consistent with the effect of carrying curvature information across changing temporal baselines.

\phantomsection
\subsection{Numerical Checks of the Drift-Recovery Results}
\label{sec:numerical-verification}

We conclude the experiments with two numerical checks of the results in Section~\ref{sec:properties}.

\noindent\textbf{Affine-drift gradient invariance (Theorem~\ref{thm:exact-decoupling}).}
We generated synthetic sample sets $\mathcal{Y}$ and added linear drift at rates $\beta \in \{1, 10^3, 10^6\}$. Relative to the drift-free baseline, the spatial gradient difference $\|\vct{g} - \vct{g}^{\mathrm{stat}}\|_2$ was $1.5 \times 10^{-13}$ at $\beta = 1$ and $5.5 \times 10^{-8}$ at $\beta = 10^6$, while the recovered temporal coefficient agreed with $\gamma^{\mathrm{stat}}+\beta$ to 10 significant digits across the tested configurations.

\noindent\textbf{Nonlinear-drift bias (Theorem~\ref{thm:nonlinear-drift}).}
Under quadratic drift $\mu(t) = \tfrac{1}{2} L_\mu t^2$, we evaluated $|\gamma - \gamma^{\mathrm{stat}} - \mu'(t_k)|$ as a function of the active-set temporal span $\tau_{\max}$. Over an eightfold variation in $\tau_{\max}$, the empirical bias was approximately linear in $\tau_{\max}$, with correlation $r > 0.998$, consistent with the bound for the recovered temporal coefficient in Theorem~\ref{thm:nonlinear-drift}.

\phantomsection
\section{Conclusion}
\label{sec:conclusion}

We have introduced TOBYQA, a time-augmented framework for derivative-free optimization under noise and temporal drift.

The analysis establishes solvability of the augmented KKT system under a full-column-rank condition on $X_{\mathrm{aug}}$. Because the composite kernel $K=A+\rho^{-1}I_m$ is positive definite, this condition is sufficient for a unique solution without the strict interpolation poisedness required by classical minimum-Frobenius-norm models (Proposition~\ref{prop:solvability}). The same recovery operator satisfies $\Pi X_{\mathrm{aug}}=I_{n+2}$. Consequently, under affine observation drift, the recovered spatial gradient is invariant to the drift magnitude, while under $C^2$ nonlinear drift its bias is bounded by $O(L_\mu\tau_{\max}^2)$ and therefore depends on the temporal span of the active sample set (Theorems~\ref{thm:exact-decoupling} and~\ref{thm:nonlinear-drift}).

The kernel also admits a statistical interpretation. Under a rotation-invariant Gaussian curvature prior, Powell's residual kernel $A_{ij}=\tfrac{1}{2}(\vct{d}_i^\top\vct{d}_j)^2$ is proportional to the covariance of the omitted quadratic forms, so the soft-constrained KKT estimate coincides with the generalized least-squares estimator under the resulting residual model (Proposition~\ref{prop:kernel-stat}). Under the probabilistically fully-linear framework and the regularization confinement condition, the expected oracle complexity above the statistical noise floor is $O(L_g(f(\vct{x}_0)-f_{\mathrm{low}})\varepsilon^{-2})$, matching the lower bound for first-order nonconvex optimization (Theorem~\ref{thm:main-complexity}).

The experiments cover $65$ benchmark problems across three dimensions ($n \in \{6, 10, 20\}$) and seven temporal drift regimes. TOBYQA has the highest reported solve rates in these tests, with the largest differences occurring under additive drift. The ablations show lower solve rates when either drift compensation or the curvature prior kernel is removed. The budget-resolved comparison also shows that the effect of fitting explicit quadratic curvature depends on both dimension and drift channel: it is less effective in low dimension and under nonlinear drift, but can be beneficial in higher dimension under the tested static and affine-drift settings. This pattern is not explained by the identifiability ratio of the quadratic coefficients alone.

Several questions remain open. The dimension-dependent behavior of the interpolated curvature variant calls for a criterion that predicts when curvature estimation is useful at the budget $m=2n+1$. The present framework assumes a static latent objective with a drifting observation channel; extending it to $f(\vct{x},t)$ would require a model for the motion of $\vct{x}_*(t)$ and corresponding tracking-error bounds. Polynomial time columns $[\vct{\theta},\vct{\theta}^2,\dots,\vct{\theta}^p]$ could be used to represent polynomial drift, but the conditioning of the enlarged system and its effect on recovery error remain to be analyzed. Extensions to bound constraints $\vct{l}\le\vct{x}\le\vct{u}$ and general linear constraints provide another direction for further work.

\phantomsection
\bibliographystyle{siamplain}
\bibliography{references}

\end{document}